\documentclass[a4paper,11pt]{amsart} 

\usepackage{amsmath,amsfonts,amsthm,amssymb}
\usepackage{setspace}
\usepackage{Tabbing}
\usepackage{fancyhdr}
\usepackage{lastpage}
\usepackage{chngpage}
\usepackage{soul,color}
\usepackage{xcolor}
\usepackage{wrapfig}
\usepackage{graphicx}
\usepackage{tikz}								
\usepackage{float}
\usepackage{amssymb}						
\usepackage{calc}								
\usepackage{cite}									
\usepackage[shortlabels]{enumitem} 	
\usepackage[hmarginratio=1:1]{geometry}
\usepackage{tensor}							

\theoremstyle{plain}
\newtheorem{theorem}{Theorem}[section]
\newtheorem{lem}[theorem]{Lemma}
\newtheorem{prop}[theorem]{Proposition}
\newtheorem{cor}[theorem]{Corollary}

\theoremstyle{remark}
\newtheorem{rem}[theorem]{Remark}

\theoremstyle{plain}
\newtheorem*{introthm}{Theorem}

\theoremstyle{definition}
\newtheorem{defn}{Definition}[section]

\newtheorem{assump}{Assumption}[section]

\theoremstyle{remark}

\newcommand{\A}{\mathbb{A}}

\newcommand{\Ax}{\mathbb{A}^{\times}}
\newcommand{\Cx}{\mathbb{C}^{\times}}

\newcommand{\C}{\mathbb{C}}

\newcommand{\W}{\mathbb{W}}

\newcommand{\Sch}{\mathcal{S}}

\newcommand{\Pe}{\mathcal{P}}
\newcommand{\Qe}{\mathcal{Q}}
\newcommand{\V}{\mathcal{V}}

\newcommand{\B}{\mathcal{B}}

\newcommand{\I}{\mathcal{I}}

\newcommand{\Ga}{\mathcal{G}}

\newcommand{\GLt}{\operatorname{GL}_{2}}

\newcommand{\GL}{\operatorname{GL}}
\newcommand{\SLt}{\operatorname{SL}_{2}}

\newcommand{\Msymt}{\operatorname{M}^{\mathrm{sym}}_{2}}

\newcommand{\Msym}{\operatorname{M}^{\mathrm{sym}}}

\newcommand{\GSpf}{\operatorname{GSp}_{4}}
\newcommand{\PGSpf}{\operatorname{PGSp}_{4}}

\newcommand{\GSp}{\operatorname{GSp}}
\newcommand{\Spf}{\operatorname{Sp}_{4}}

\newcommand{\Sp}{\operatorname{Sp}}

\newcommand{\GO}{\operatorname{GO}}
\newcommand{\G}{\operatorname{G}}
\newcommand{\GSO}{\operatorname{GSO}}
\newcommand{\Or}{\operatorname{O}}
\newcommand{\SO}{\operatorname{SO}}

\newcommand{\Dx}{D^{\times}}

\newcommand{\Kx}{K^{\times}}

\newcommand{\Ex}{E^{\times}}
\newcommand{\Fx}{F^{\times}}
\newcommand{\Bx}{B^{\times}}
\newcommand{\Lx}{L^{\times}}

\newcommand{\inv}{^{\ast}}

\newcommand{\Tr}{\operatorname{Tr}}

\newcommand{\N}{\operatorname{N}}

\newcommand{\Hom}{\operatorname{Hom}}
\newcommand{\im}{\operatorname{Im}}

\newcommand{\Res}{\operatorname{Res}}
\newcommand{\Ind}{\operatorname{Ind}}
\newcommand{\sgn}{\operatorname{sgn}}
\newcommand{\isom}{\cong}
\newcommand{\pr}{^{\prime}}

\newcommand{\abs}[1]{\lvert #1 \rvert}	
\newcommand{\disc}{\operatorname{disc}}
\newcommand{\Ad}{\operatorname{Ad}}
\newcommand{\std}{\operatorname{std}}
\newcommand{\Vol}{\operatorname{Vol}}
\newcommand{\AI}{\operatorname{AI}}
\newcommand{\supp}{\operatorname{supp}}

\newcommand{\tran}[2]{\tensor[^t]{#1}{^{#2}}}
\newcommand{\dgroup}[1]{\tensor[^{L}]{#1}{}}

\newcommand{\du}{^{\vee}}

\newcommand{\bk}{\backslash}
\newcommand{\BC}{\operatorname{BC}}

\newcommand{\Au}{\mathcal{A}}
\newcommand{\Ao}{\mathcal{A}_{0}}

\newcommand{\Gal}{\operatorname{Gal}}

\newcommand{\weil}{\operatorname{\boldsymbol\omega}}
\newcommand{\Weil}{\operatorname{\boldsymbol\Omega}}

\newcommand{\Set}{\mathfrak{S}}

\newcommand{\redstar}{\textcolor{red}{\bigstar}}
\newcommand{\bluestar}{\textcolor{blue}{\bigstar}}

\newcommand{\emb}{\hookrightarrow}
\newcommand{\emblong}{\ensuremath{\lhook\joinrel\relbar\joinrel\rightarrow}}
\newcommand{\arrup}[2]{\mathrel{\mathop{#1}^{#2}}}

\makeatletter	

\renewcommand{\tocsection}[3]{%
  \indentlabel{\@ifnotempty{#2}{\ignorespaces#1 #2.~}}#3\dotfill}

\let\oldtocsection=\tocsection

\renewcommand{\tocsection}[2]{\hspace{0em}\oldtocsection{#1}{#2}}

\newcommand{\makeandyabstract}{
\normalfont\normalsize
\@setabstract
}

\makeatother	

\newcommand{\ReportTitle}{A proof of the refined Gan--Gross--Prasad conjecture\\ for non-endoscopic Yoshida lifts}
\newcommand{\ShortTitleCaps}{THE REFINED GAN--GROSS--PRASAD CONJECTURE FOR YOSHIDA LIFTS} 

\newcommand{\ReportAuthorName}{Andrew J.~Corbett}

\newcommand{\ReportDate}{29$^{\mathrm{th}}$ January 2016}

\title[\ShortTitleCaps]{\Large\ReportTitle}
\author{\ReportAuthorName}
\date{\ReportDate}
\address{School of Mathematics, University of Bristol, Bristol, BS8 1TW, United Kingdom}
\email{andrew.corbett@bristol.ac.uk}

\numberwithin{equation}{section}

\begin{document}

\pagenumbering{arabic}
\setcounter{page}{1}

\begin{abstract}
We prove a precise formula relating the Bessel period of certain automorphic forms on ${\rm GSp}_{4}(\mathbb{A}_{F})$ to a central $L$-value. This is a special case of the refined Gan--Gross--Prasad conjecture for the groups $({\rm SO}_{5},{\rm SO}_{2})$ as set out by Ichino--Ikeda \cite{ichino_ikeda} and Liu \cite{liu}. This conjecture is deep and hard to prove in full generality; in this paper we succeed in proving the conjecture for forms lifted, via automorphic induction, from ${\rm GL}_{2}(\mathbb{A}_{E})$ where $E$ is a quadratic extension of $F$. The case where $E=F\times F$ has been previously dealt with by Liu \cite{liu}.
\end{abstract}

\maketitle

\tableofcontents

\section{Introduction}

The aim of this paper is to prove a special case of a deep conjectural relation between periods of automorphic forms and central values of $L$-functions. An early prototype of such a result is due to Waldspurger \cite{waldspurger}, who computed toric integrals of automorphic forms on $\GLt$ to be an `Euler-product' of local integrals scaled by a global constant of certain $L$-values. Soon after, Gross--Prasad \cite{gp1} made a series of fascinating conjectures relating periods of $\SO_{n+1}\times\SO_{n}$-forms along $\SO_{n}$ (embedded diagonally) to central $L$-values -- the case $n=2$ is implied by the work of Waldspurger. These conjectures were extended to include all classical groups by Gan--Gross--Prasad \cite{ggp}.


In their original form, the Gross--Prasad conjectures omit a precise description of the factorisation of the global automorphic period. However, a recent work of Liu \cite{liu}, extending that of Ichino--Ikeda \cite{ichino_ikeda}, offers a refined conjecture by giving a precise conjectural formula for the Bessel period of a wide family of automorphic forms in terms of the central values of certain $L$-functions. In its full generality, Liu's conjecture appears out of reach of our current methods, even for specific groups. Nevertheless, one can try to prove special cases of it; Liu himself proved his conjecture in the case of \textit{endoscopic} automorphic forms on $\GSpf$ \cite{liu}, motivated by Prasad--Takloo-Bighash \cite{prasad_takloo-bighash}. These endoscopic forms are classically known as Yoshida lifts and essentially correspond to lifts from $\GLt\times\GLt$.

In this paper we prove such a formula for the \textit{non-endoscopic Yoshida lifts}: the automorphic forms on $\GSpf$ lifted from the \textit{non-split} orthogonal group $\GO_{4}$ (that is, the underlying quadratic space defining $\GO_{4}$ has non-square discriminant).  Making use of exceptional isomorphisms, we see that such forms are obtained by automorphic induction from $\GLt(E)$ where $E$ is a quadratic extension of the base field $F$. (Liu's result covers the \textit{split} case where $E=F\times F$.) For our proof we require both a much finer analysis of the four-dimensional quadratic spaces governing $\GO_{4}$ (of non-square discriminant) and a more detailed construction of the automorphic representations of this group than that found in \cite{liu}. This analysis provides a notable diversion from Liu's method, especially in the final deduction of our explicit formula \S \ref{sec_final}.

Before describing our results in more detail we also remark on a conjecture of B\"ocherer \cite{bocherer_conj} (see also \cite{saha_bc}). In this work B\"ocherer formulates an equality between sums of Fourier coefficients (indexed by ideal classes of a fixed quadratic field $K$) of Siegel modular forms and certain $L$-values. The present paper considers the Bessel period of an automorphic form on $\GSpf(\A)$; if the form in question is the ad\`elisation of a Siegel modular form then (by \cite{furusawa} for example) one computes the Bessel period to be precisely the Fourier coefficients that B\"ocherer considered. Thus our result provides a proof of (a refinement of) B\"ocherer's conjecture for non-endoscopic Yoshida lifts.

\subsection{The Bessel period}

Let $F$ be a (totally real) number field with ad\`ele ring $\A=\A_{F}$. We consider the refined Gan--Gross--Prasad conjecture for the groups $(\SO_{5},\SO_{2})$. In this case we extend $\SO_{2}$ to the \textit{Bessel subgroup} $R=U\rtimes\SO_{2}$, with $R\emb\SO_{5}$, where $U$ is a certain unipotent subgroup of $\SO_{5}$. The conjecture describes the explicit form of a period integral of automorphic forms on $\SO_{5}\times R$ along the (diagonally embedded) subgroup $R$. Our approach to the problem makes use of the exceptional isomorphisms
\begin{equation*}
\SO_{5}\isom\PGSpf\quad {\rm and} \quad\SO_{2}\isom\Res_{K/F}\Kx/\Fx
\end{equation*}
where $K$ is a quadratic field extension of $F$.

More specifically, let $\chi$ be a unitary Hecke character of $\Ax_{K}$, simultaneously thought of as a character of $\SO_{2}(F)\bk \SO_{2}(\A)$, and let $\pi$ be an irreducible, cuspidal automorphic representation of $\GSpf(\A)$ in the space of cusp forms $\V_{\pi}$. Impose the central character condition that $\pi\otimes \chi \vert_{\Ax} =1$. Additionally, make a (standard and inconsequential) choice of automorphic character $\psi$ of $U$ so that $\psi\boxtimes\chi$ is an automorphic character of $R$. We then define the \textit{$\chi$-Bessel period} of $\varphi\in\V_{\pi}$ to be the absolutely convergent integral
\begin{equation}\label{eq_bessel_intro}
\Pe(\varphi,\chi)=\displaystyle\int_{\Ax R(F)\bk R(\A)}\varphi(g)\,(\psi\boxtimes\chi)(g)\,dg\,.
\end{equation}
This integral defines an element of $\Hom_{R(\A)}(\pi\otimes(\psi\boxtimes\chi),\C)$. The unrefined conjecture claims that there exists some vector $\varphi^{\ast}$ in (the Vogan $L$-packet of) $\pi$ such that
\begin{equation*}
\Pe(\varphi^{\ast},\chi)\neq 0 \quad\Longleftrightarrow\quad L(1/2,\pi\boxtimes \chi)\neq 0
\end{equation*}
where $\Pe(\varphi^{\ast},\chi)$ may be defined for more general elements $\varphi^{\ast}$ of the Vogan $L$-packet in a similar way to \eqref{eq_bessel_intro}. It is this unrefined dependence which we make explicit.

To discuss the local side, assume the factorisations $\pi=\otimes_{v}\pi_{v}$ ; $\chi=\otimes_{v}\chi_{v}$ ; $\psi=\otimes_{v}
\psi_{v}$ and suppose that $\varphi=\otimes_{v}\varphi_{v}$. Associated to this data, we follow Liu in defining
\begin{equation*}
\alpha^{\natural}(\varphi_{v},\chi_{v})\in\Hom_{R(F_{v})}(\pi_{v}\otimes(\psi_{v}\boxtimes\chi_{v}),\C)
\end{equation*}
at each place $v$ to be an integral over local matrix coefficients (see \S \ref{sec_local}). Roughly speaking -- up to a normalisation constant (see \eqref{eq_nomalise_complicated}) -- the integral defining $\alpha^{\natural}(\varphi_{v},\chi_{v})$ is equal to
\begin{equation*}
\int_{\Fx_{v}\bk R(F_{v})}\B_{\pi_{v}}( \pi_{v}(g_{v})\varphi_{v},\bar{\varphi}_{v})\,(\chi_{v}\boxtimes\psi_{v})(g_{v})\, dg_{v}
\end{equation*}
where $\B_{\pi_{v}}$ is a local unitary pairing for $\pi_{v}$. The foundation on which Liu is able to generalise the refined conjecture is the regularisation of these integrals. They are shown to converge absolutely and a natural normalisation is found such that $\alpha^{\natural}(\varphi_{v},\chi_{v})=1$ for almost all places $v$ \cite[Theorem 2.1 \& 2.2]{liu}. We may thus make sense of the infinite product $\prod_{v}\alpha^{\natural}(\varphi_{v},\chi_{v})$. The refined Gan--Gross--Prasad conjecture then asks for the constant of proportionality between this product of local factors and the square of the absolute value of the Bessel period.

\subsection{Lifted representations}\label{sec_intro_lifts} The representations of $\SO_{5}(\A)\isom\PGSpf(\A)$ are precisely those representations of $\GSpf(\A)$ with trivial central character. We consider a class of representations of $\PGSpf(\A)$ which are lifted from representations of the group $\GO_{4}(\A)$, when $\GO_{4}$ is non-split, via the theta correspondence for $(\GO_{4},\GSpf)$ -- we call such lifted representations the \textit{non-endoscopic Yoshida lifts}. The domain of this lift comprises of the representations of $\GO_{4}(\A)$ (of trivial central character); these are uniquely determined by representations of $\Dx(\A_{E})$ for a canonical choice of quadratic extension $E/F$ and quaternion algebra $D$ over $F$. Thus, via Jacquet--Langlands transfer, one may view a non-endoscopic Yoshida lift $\pi$ as being of the form $\pi=\AI(\pi')$: the automorphic induction, to $\GSpf(\A)$, of a representation $\pi'$ of $\GLt(\A_{E})$.


\subsection{Main result} We refer the reader to Theorem \ref{thm_andy_thm1} for a more precise statement of our result. To simplify notation here assume the following decompositions for both the Petersson inner product $\B_{\pi}$ on $\pi$ and the Tamagawa measure $dg$ on $\Ax\bk R(\A)$:
\begin{equation}\label{eq_decompositions}
\B_{\pi}=\prod_{v}\B_{\pi_{v}}\,,\qquad dg=\prod_{v}dg_{v}\,
\end{equation}
where $\B_{\pi_{v}}$ and $dg_{v}$ are the local factors used to define $\alpha^{\natural}(\varphi_{v},\chi_{v})$.
\begin{introthm}
Let $\pi=\AI(\pi\pr)$ be a non-endoscopic Yoshida lift to $\PGSpf(\A)$, as per \S \ref{sec_intro_lifts}, where $\pi\pr$ is an irreducible, cuspidal automorphic representation of $\GLt(\A_{E})$ with trivial central character. Let $K$ be a quadratic field extension of $F$ such that $\SO_{2}\isom\Kx/\Fx$. Let $\chi$ be a unitary Hecke character of $\Ax_{K}$ with $\chi\vert_{\Ax}=1$, then $\chi$ is simultaneously an automorphic character of $\SO_{2}(\A)$. Denote by $\chi_{K/F}$ the quadratic character associated to $K$ by class field theory. Assume the choices of \eqref{eq_decompositions} and that the local integrals $\alpha^{\natural}(\varphi_{v},\chi_{v})$ are \textit{properly} normalised (as in Definition \ref{def_properly_normalised}). Then for a cusp form $\varphi=\otimes_{v}\varphi_{v}$ in the space associated to $\pi$ we have

\begin{equation*}
\abs{\Pe(\varphi,\chi)}^{2}=\dfrac{1}{4}\, \dfrac{\zeta_{F}(2)\,\zeta_{F}(4)\,L(1/2,\pi\boxtimes\chi)}{L(1,\pi,\Ad)\,L(1,\chi_{K/F})}\,\prod_{v}\alpha^{\natural}(\varphi_{v},\chi_{v})\,.
\end{equation*}
\end{introthm}

\subsection{Remarks}

The case where $E=F\times F$ is dealt with by Liu \cite[\S 4]{liu}. Liu's theorem determines the Bessel period attached to an automorphic form on $\GSpf$ which is a lift from $\GLt\times\GLt$. These lifts are precisely the  \textit{endoscopic} representations of $\GSpf$. Moreover, Qui has proved a formula for $\abs{\Pe(\varphi,\chi)}^{2}$ when $\pi$ is in the \textit{nontempered} cuspidal spectrum of $\SO_{5}$ (see \cite{qiu}). This is achieved by considering the so-called Saito--Kurukawa and Soudry lifts.

Following these two works, this paper uses the functorial lift from $\GLt(E)$ to give a wide class of nonendoscopic, tempered, cuspidal automorphic representations of $\PGSpf$ that conform to the refined Gan--Gross--Prasad conjecture. Further works on attempting to prove such a formula in general have been approached by using tools such as relative trace formulae (see \cite{furusawa_mem} for example).

The assumption that $F$ is a totally real number field is needed only to permit the application of a result of \cite{gan_ichino} on the Petersson inner product of a theta lift; they, in turn, only require this assumption to use the Siegel--Weil formula in their calculation.

Finally we would like to highlight the occurrence of the constant $1/4$ in our formula, to be compared with the constant $1/8$ appearing in \cite{liu}. This falls in line with the general conjecture of Liu \cite{liu} in that it relates precisely to the (conjectural) Arthur parameters of $\pi$ and $\chi$ (as first pointed out by Ichino--Ikeda \cite[\S 2]{ichino_ikeda} and then by Gan--Ichino \cite[Remark 1.2]{gan_ichino}). Specifically, the constant should be $\frac{1}{\abs{\mathcal{S}_{\pi}}\abs{\mathcal{S}_{\chi}}}$ where $\mathcal{S}_{\pi}$ (resp. $\mathcal{S}_{\chi}$) is the centraliser of the image of the Arthur parameter of $\pi$ (resp. $\chi$); note that in our case we trivially have $\abs{\mathcal{S}_{\chi}}=2$. The discrepancy of $1/2$ between our result and that of \cite{liu} is supported by the observation that
\begin{equation*}
\abs{\mathcal{S}_{\pi}}=\left\lbrace\begin{array}{cl}\vspace{0.05in}
4 & {\rm if}\,\, E=F\times F,\\
2 & {\rm if}\,\, E=F(\sqrt{e}\,)\,\, {\rm for\,\,some}\,\,e\not\in(\Fx)^{2}.
\end{array}\right.
\end{equation*}
It is interesting to see this factor arise naturally due to the structure of the representations of $\GO_{4}(\A)$: in \cite{liu} the Bessel period boils down to twice the period considered by Waldspurger \cite{waldspurger} in contrast to the single occurrence that we observe in our computation.

This paper is set out as follows: after some preliminary definitions regarding the Bessel period (\S \ref{sec_prelim}) we review the theta correspondence for $(\GO_{4},\GSpf)$ (\S \ref{sec_theta_section}) and discuss the representation theory of $\GO_{4}$ (\S \ref{sec_autm_reps_H}), explaining the lift we use and its domain. We then analyse the global (\S \ref{sec_global}) and local (\S \ref{sec_local}) periods before uniting these quantities (\S \ref{sec_final}) via a theorem of Waldspurger and proving the result at hand.

\subsection*{Acknowledgements}

The author would like to offer sincere thanks to both Yifeng Liu, for his helpful comments and discussions, and Abhishek Saha, for his valuable guidance. Thanks are also due to Katharine Thornton for her many insightful suggestions.

\section{Preliminary Discussion}\label{sec_prelim}

\subsection{Some conventions} We work over a fixed number field $F$ which we assume to be totally real. Put $\mathcal{O}$ for the ring of integers of $F$ and $\A$ for the ring of $F$-ad\`eles. Given an extension $L\supset F$ let $\A_{L}=\A\otimes_{F} L$.

If $G$ is a linear algebraic group defined over $F$ and $R$ is an $F$-algebra write $G(R)$ for the $R$-points of $G$. At a place $v$ of $F$ simplify the notation $G(F_{v})$ to $G_{v}$. Given a function $f$ on $G$, denote left and right translation by elements $g\in G$ by
\begin{equation*}
L(g)f(x)=f(g^{-1}x)\quad {\rm and}\quad R(g)f(x)=f(xg)\,.
\end{equation*}

If $S$ is a finite set of places of $F$ then introduce the following notation: $F_{S}=\prod_{v\in S}F_{v}$ and $\A^{S}=\prod_{v\not\in S}\pr F_{v}$. Note the compatibility of the products $G(F_{S})=\prod_{v\in S}G(F_{v})$ and $G(\A^{S})=\prod_{v\not\in S}\pr G(F_{v})$ meaning that we can formally identify $G(\A)=G(F_{S})G(\A^{S})$.

\subsubsection{Measures}

For an algebraic group $G$ we fix a Haar measure on $G(\A)$ by taking the Tamagawa measure $dg$ (as originally defined in \cite{weil_tam}). Let $dg_{v}$ be a specified choice of local Haar measures on $G_{v}$ for each $v$ such that $\prod_{v} dg_{v}$ is a well defined measure on $G(\A)$. By the uniqueness of Haar measures there exists a constant of proportionality $C\in\C$ such that $dg=C\prod_{v} dg_{v}$. We call such a $C$ \textit{Haar measure constant}, as in \cite{ichino_ikeda}.

\subsubsection{Automorphic representations and pairings} The space of automorphic (resp. cusp) forms on $G(\A)$ shall be denoted $\Au(G)$ (resp. $\Ao(G)$). For an irreducible, cuspidal automorphic representation $\pi$ of $G(\A)$ we denote by $\V_{\pi}$ the realisation of $\pi$ in $\Ao(G)$ and put $\omega_{\pi}$ for its central character. One has $\pi\isom\otimes_{v}\pi_{v}$ (and $\V_{\pi}\isom\otimes^{\prime}_{v}\V_{\pi_{v}}$) where at each place $v$ of $F$, $\pi_{v}$ is an irreducible, admissible, unitary representation of $G_{v}$ on $\V_{\pi_{v}}$. Let $\bar{\pi}$ denote the \textit{conjugate} representation of $\pi$ realised on the space
\[\V_{\bar{\pi}}=\left\lbrace \bar{f}\,:\,f\in\V_{\pi}\right\rbrace.\]
There is a canonical bilinear pairing $\B_{\pi}\colon\V_{\pi}\otimes\V_{\bar{\pi}}\rightarrow\C$ given by the Petersson inner product
\begin{equation*}
\B_{\pi}(f,\tilde{f})=\displaystyle\int_{Z_{G}(\A)G(F)\bk G(\A)}f(g)\tilde{f}(g)dg
\end{equation*}
where $Z_{G}$ is the maximal split torus in the centre of $G$ and $dg$ is the Tamagawa measure on $(Z_{G}\bk G)(\A)$ as always. In particular, since $\V_{\pi}$ is a complex Hilbert space and $\pi$ is unitary, one can show that $\bar{\pi}$ is isomorphic to $\pi\du$, the \textit{contragredient representation} of $\pi$ realised on the space of smooth vectors in the dual space $\V_{\pi}\du$ of $\V_{\pi}$. Moreover, any pairing on a unitary Hilbert space representation is unique up to a scalar factor. Both of these facts are corollaries to the Riesz representation theorem. Throughout, any local, irreducible, admissible representation of $G_{v}$ is always considered to be unitary.


\subsubsection{$L$-functions}\label{sec_l_functions} Given a representation $r$ of the Langlands dual group and an automorphic representation $\pi$ of $G$ we have the Langlands $L$-function $L(s,\pi,r)$. When $r$ is the standard representation of the dual group, which we assume is a subgroup of $\GL_{n}(\C)$, we write $L(s,\pi)$ for $L(s,\pi,r)$. The notation $\pi_{1}\boxtimes\pi_{2}$ denotes the (external tensor product) representation of the direct product $G_{1}\times G_{2}$, where $\pi_{i}$ are representations of the groups $G_{i}$ for $i=1,2$, respectively.

The most interesting $L$-function for us is given as follows. Let $\pi$ be an automorphic representation of $\PGSpf(\A)\isom\SO_{5}$ and let $\chi$ be a character of $\SO_{2}(F)\bk \SO_{2}(\A)$ corresponding to a Hecke character of $\Ax_{K}$ as in the introduction. Then we consider the $\SO_{5}\times\SO_{2}$ $L$-function $L(s, \pi\boxtimes\chi)$. However, other authors interpret this $L$-function as:

\begin{itemize}\vspace{0.05in}
\item the $\GSpf\times\GLt$ $L$-function $L(s,\pi\boxtimes\AI(\chi))$, where $\AI(\chi)$ is the automorphic induction of $\chi$ from $\Ax_{K}$ to $\GLt(\A)$, or,\vspace{0.05in}
\item the $\GSpf(K)$ $L$-function $L(s,\BC(\pi)\otimes\chi)$, where $\BC(\pi)$ is the base change of $\pi$ from $\GSpf(\A_{F})$ to $\GSpf(\A_{K})$.\vspace{0.05in}
\end{itemize}
Each of these representations arises due to a functorial transfer from the original representation $\pi\boxtimes\chi$. The characteristic property of such a transfer implies that these $L$-functions are indeed all equal.

Other notation includes: $\zeta_{F}$, the Dedekind zeta function for a number field $F$, and $\chi_{K/F}$ which always denotes the quadratic character of $\Kx$ given by class field theory. Note that for any Hecke character $\chi$ of $\Ax_{K}$, the adjoint $L$-function is trivially $L(s,\chi,\Ad)=L(s,\chi_{K/F})$.

\subsubsection{Quadratic spaces}\label{sec_quad_spaces} Let $(V,q)$ be a quadratic space over $F$ of even dimension $2m$ (we always assume such a $V$ is non-degenerate). The quadratic form $q$ corresponds to a symmetric matrix $S_{q}\in\Msym_{2m}(F)$ such that $q(v)=\tran{v}{}S_{q}v$ for $v\in V$. One defines the \textit{discriminant} of $V$ to be $\disc V = (-1)^{m}\,\det S_{q}$ and the associated \textit{discriminant algebra} as
\begin{equation}\label{eq_discalg}
K_{V}=\left\lbrace\begin{array}{ll}\vspace{0.05in}
 F(\sqrt{\disc V}\,)&{\rm if}\,\, \disc V\not\in (\Fx)^{2}\\
 F\times F &{\rm if}\,\, \disc V\in (\Fx)^{2}\,.
\end{array}\right.\vspace{0.05in}
\end{equation}
We intend to study the \textit{orthogonal similitude group} of $V$:

\begin{equation*}
\begin{array}{rcl}\vspace{0.1in}
\GO(V)&=&\left\lbrace\, g\in\GL(V)\, :\, q(gv)=\lambda(g)q(v)\,\,\forall\,v\in V\,\right\rbrace\\\vspace{0.05in}
&=&\left\lbrace\, g\in\GL_{2m}(F)\, :\, \tran{g}{}S_{q}g=\lambda(g)S_{q}\,\right\rbrace
\end{array}
\end{equation*}
where $\lambda\colon\GO(V)\rightarrow \Fx$ is the \textit{similitude character}. One observes that $(\det g)^{2} = \lambda(g)^{2m}$, so there is a natural \textit{sign character} on $\GO(V)$:
\begin{equation*}
\sgn\colon g\longmapsto \det g/\lambda(g)^{m}\,\in\mu_{2}
\end{equation*}
where $\mu_{2}=\mu_{2}(F)$. We define the connected component of $\GO(V)$ to be the normal subgroup $\GSO(V)=\ker(\sgn)$ which sits in the exact sequence

\begin{equation*}
1\longrightarrow\GSO(V)\longrightarrow\GO(V)\arrup{\longrightarrow}{\sgn\,\,}\mu_{2}\longrightarrow 1\,.
\end{equation*}
Similarly, if one defines the classical \textit{orthogonal group} $\Or(V)=\ker(\lambda)$, then the \textit{special orthogonal group} $\SO(V)$ is found in the exact sequence

\begin{equation*}
1\longrightarrow\SO(V)\longrightarrow\Or(V)\arrup{\longrightarrow}{\det\,\,}\mu_{2}\longrightarrow 1
\end{equation*}
where $\det=\sgn$ here. When $\dim V=4$ we see later in \S \ref{sec_fourdim} that the sign character is surjective and we exhibit a natural choice of representatives for $\GO(V)/\GSO(V)$. In essence, there is a unique element $\iota\in\GO(V)$ with
\begin{equation}\label{eq_first_mention_iota}
\lambda(\iota)=1\,;\qquad\iota^{2}=1\,;\qquad \det\iota=-1\,.
\end{equation}
We are then able to fix a splitting such that $\mu_{2}$ is identified with the subgroup of $\GO(V)$ generated by $\iota$. In particular we arrive at the decomposition $\GO(V) = \GSO(V)\rtimes \mu_{2}$.

\begin{rem}
For an $F$-algebra $A$, the above comments apply more generally to the exact sequence
\begin{equation*}
1\longrightarrow \GSO(V)(A)\longrightarrow \GO(V)(A)\arrup{\longrightarrow}{\sgn\,\,}\mu_{2}(A)\longrightarrow 1\vspace{0.07in}
\end{equation*}
where the $A$ points of $\GSO(V)$ coincide with the kernel of the sign function on $\GO(V)(A)$. In particular we have a well defined notion of $\mu_{2}(\A)$, $\GSO(V)(\A)$, $\GSO(V)_{v}$ and so on.
\end{rem}

\subsection{The Bessel period and definitions}\label{sec_defn_gps_period}

\subsubsection{$\GSpf(F)$ in coordinates}\label{sec_GSpf_in_coord} Let $W=F^{4}$ and endow $W$ with an antisymmetric bilinear form $(\cdot,\cdot)_{W}$ so that $W$ becomes a four-dimensional symplectic vector space over $F$. In the coordinates of $F^{4}$ one may choose
\[(u,v)_{W}=\tran{u}{}
\begin{pmatrix}
0&1_{2}\\
-1_{2}&0
\end{pmatrix}v\,\]
where $1_{2}$ is the $2\times 2$ identity matrix. Setting $W_{1}=F^{2}$ then $W=W_{1}\oplus W_{1}\du$ gives a complete polarisation of $W$ such that $W_{1}\du$ is identified with the dual space of $W_{1}$ under the form $(\cdot,\cdot)_{W}$. Recall the definition for the \textit{symplectic similitude group}:
\[\GSpf(F)=\GSp(W)=\left\lbrace \, g\in\GL(W)\,:\,(gu,gv)_{W}=\lambda(g)(u,v)_{W}\quad\forall\,\,u,v\in W\right\rbrace\,\]
where $\lambda(g)\in\Fx$. We use $\lambda$ for the similitude character of any similitude group.

\subsubsection{The torus}\label{sec_the_torus}

Fix a choice of anisotropic, symmetric matrix

\begin{equation*}
S=\begin{pmatrix}a&b/2\\
b/2&c
\end{pmatrix}
\in\Msymt(F)\vspace{0.07in}
\end{equation*}
to represent the quadratic form $q_{S}(v)=\tran{v}{}Sv$ for $v\in W_{1}$. Then $(W_{1},q_{S})$ is a two-dimensional quadratic space over $F$ of (scaled) discriminant
\begin{equation*}
d=-4\det S=b^{2}-4ac\,.
\end{equation*}
By the anisotropy of $S$ (that $q_{S}(v)=0\Rightarrow v=0$) it is clear that $d$ is not a square in $F$. Hence the discriminant algebra $K_{W_{1}}=F(\sqrt{d}\,)$ is a quadratic field extension of $F$. Fix the notation $K=K_{W_{1}}$. We consider a maximal, non-split torus in $\GLt(F)$ given by the orthogonal group
\begin{equation*}
T=T_{S}=\left\lbrace g\in\GL_{2}(F)\, \colon\, \tran{g}{}Sg=(\det g)S\right\rbrace=\GSO(W_{1})\,.
\end{equation*}
One has the isomorphism $T\isom\Res_{K/F}\Kx$ of algebraic groups over $F$. Specifically, one shows that
\begin{equation*}
T(F)=\left\lbrace \, x+y \left(\begin{smallmatrix}
b/2&c\\
-a&-b/2
\end{smallmatrix}\right)\,:\,x,y\in F\right\rbrace^{\times}\vspace{0.07in}
\end{equation*}
and defines an isomorphism $T(F) \rightarrow \Kx=F(\sqrt{d}\,)^{\times}$ by

\begin{equation*}\begin{array}{c}
x+y \left(\begin{smallmatrix}
b/2&c\\
-a&-b/2
\end{smallmatrix}\right)\longmapsto x+y\frac{\sqrt{d}\,}{2}\,.
\end{array}
\end{equation*}

\subsubsection{The Bessel subgroup} Consider the following subgroups of $\GSpf(F)$:

\begin{itemize}\vspace{0.05in}

\item Let $U$ be the \textit{unipotent radical} stabilising the flag $\lbrace 0 \rbrace \subset W_{1} \subset W$; explicitly we have
\begin{equation*}
U=\left\lbrace u(A)=\begin{pmatrix}
1_{2}&A\\
0&1_{2}
\end{pmatrix}\, \colon\, A\in \Msymt(F)\right\rbrace\,.\vspace{0.07in}
\end{equation*}
All elements of $U$ have similitude $\lambda(u(A))=1$. We also identify $U$ with the space of symmetric $F$-linear maps $W_{1}\du\rightarrow W_{1}$. Taking the standard additive character
\begin{equation}\label{eq_std_char}
\psi\colon F\bk\A\longrightarrow\Cx\,,
\end{equation}
we define a character $\psi_{M}$ of $U(F)\bk U(\A)$, for a matrix $M\in\Msymt(F)$, by
\begin{equation}\label{eq_uni_char_def}
\psi_{M}(u(A))=\psi(\Tr(MA))\,.
\end{equation}
All characters of $U$ arise in this way for some $M$.\vspace{0.05in}

\item One has an embedding $T\emb\GSp(W)$ by mapping $g\in T$ to

\begin{equation*}
\hat{g}=\begin{pmatrix}
g&\\
&(\det g)\, \tran{g}{-1}
\end{pmatrix}\in \GSp(W)\,.\vspace{0.07in}
\end{equation*}
This element has similitude factor $\lambda(\hat{g})=\det g$. Moreover if $u\in U$ then $ug=gu$.\vspace{0.05in}

\item The \textit{Bessel subgroup} of $\GSpf(F)$ is then the semidirect product
\begin{equation*}
R=U\rtimes T\,.
\end{equation*}
\end{itemize}

\subsubsection{The Bessel period}\label{sec_bessel_period} Let $\pi$ be an automorphic representation of $\GSpf(\A)$. All automorphic representations of the abelian group $T(\A)$ are given by characters
\begin{equation*}
\chi\colon T(F)\bk T(\A)\longrightarrow\Cx\,,
\end{equation*}
of which we now fix a $\chi$ such that $\omega_{\pi}\cdot \chi \vert_{\Ax} =1$. We shall simultaneously think of $\chi$ as a character of $\Kx \bk \Ax_{K}$. For $\varphi_{\pi}\in\V_{\pi}$, the \textit{Bessel period} of $\varphi_{\pi}$ (with respect to $\chi$) is defined by the period integral

\begin{equation}\label{eq_bessel_def}
\Pe(\varphi_{\pi},\chi)=\displaystyle\int_{\Ax T(F)\bk T(\A)}\int_{U(F)\bk U(\A)}\varphi_{\pi}(u\hat{g})\,\chi(g)\,\psi^{-1}_{S}(u)\,du\,dg
\end{equation}
where $du$ and $dt$ are the Tamagawa measures on $U(\A)$ and $\Ax\bk T(\A)$ respectively. We realise $\Ax$ as the scalar matrices in the domain of integration $\Ax R(F)\bk R(\A)$.

\subsection{Notation for groups}\label{sec_groups_notation}

For a fixed four-dimensional quadratic space $V$ over $F$ and the four-dimensional symplectic vector space $W=F^{4}$ (from \S \ref{sec_GSpf_in_coord}) assign the notation
\begin{equation*}
\begin{array}{lll}\vspace{0.05in}
G=\GSp(W)&H=\GO(V)&H^{0}=\GSO(V)\\
G_{1}=\Sp(W)&H_{1}=\Or(V)&H_{1}^{0}=\SO(V)
\end{array}
\end{equation*}
which will be used freely throughout. Also define the groups
\begin{equation*}
Y=\G(\Sp(W)\times\Or(V))=\lbrace\, (g,h)\in \GSp(W)\times\GO(V)\,:\, \lambda(g)=\lambda(h)\,\rbrace
\end{equation*}
and
\begin{equation*}
G^{+}=\left\lbrace\,g\in G\,:\,\lambda(g) = \lambda(h)\,\, \mbox{for some}\,\, h\in H\,\right\rbrace.
\end{equation*}

\section{The Theta Correspondence for $(\GO_{4},\GSpf)$}\label{sec_theta_section}

This section is devoted to constructing certain representations of $\GSpf$ from representations of $\GO_{4}$ both locally and globally.

\subsection{The local theta correspondence}\label{sec_local_theta}

Let $v$ be a place of $F$ and omit the subscript $v$ from the notation in this section ($F=F_{v}$, $G=G(F_{v})$, $W=W\otimes_{F}F_{v}$ and so on). Define the space $\W=W\otimes V$ which is given the symplectic form $(\cdot\,,\cdot)_{\W}=(\cdot\,,\cdot)_{W}\otimes(\cdot\,,\cdot)_{V}$. Then groups $G_{1}$ and $H_{1}$ form a \textit{reductive dual pair} as subgroups of $\Sp(\W)$. The polarisation of $W=W_{1}\oplus W_{1}\du$ induces a polarisation
\begin{equation*}
\W = (W_{1}\otimes V)\oplus (W_{1}\du\otimes V)
\end{equation*}
on which we make some remarks:

\begin{itemize}
\item Having chosen the natural basis for $W$ we may identify $W_{1}\du\otimes V \isom V^{2}$.\vspace{0.05in}
\item There is an isomorphism $W_{1}\du\otimes V\isom \Hom_{F}(W_{1},V)$.
\end{itemize}
(These comments are also relevant in the global setting, considering the ad\`elic points of the above spaces.)

Choose a non-trivial additive character $\psi$ of $F$ by taking it to be a local component of the standard (additive) ad\`elic character \eqref{eq_std_char}. Let $\weil=\weil_{\psi}$ be the \textit{Weil representation} of $G_{1}\times H_{1}$, with respect to $\psi$, which may be extended to a representation of $Y$ as in \cite[p.~82]{harris_kudla}. We realise $\weil$ in the space of Schwartz functions $\V_{\weil}=\Sch(V^{2})$ where $Y$ acts as follows. For $(g,h)\in G_{1}\times H_{1}$ and $\phi\in \Sch(V^{2})$:

\begin{equation}\label{eq_schro_mod}
\begin{array}{rcl}\vspace{0.1in}
\weil(1,h)\,\phi(x)&=&\phi(h^{-1}x)\\\vspace{0.1in}	
\weil(J_{2},1)\,\phi(x)&=&\gamma_{4}\,\hat{\phi}(x)	\\\vspace{0.1in}
\weil(u(A),1)\,\phi(x)&=&\psi(\Tr(M_{x}A)) \,\phi(x)\\\vspace{0.1in}								
\weil(m(B),1)\,\phi(x)&=&\chi_{V}(\det B)\,\abs{\det B}_{F}^{2} \,\phi(xB)\\
\end{array}
\end{equation}
where the elements

\begin{equation*}
J_{2}=\begin{pmatrix}
0&1_{2}\\
-1_{2}&0
\end{pmatrix}\,;\quad
u(A)=\begin{pmatrix}
1_{2}&A\\
0&1_{2}
\end{pmatrix}\,;\quad
m(B)=\begin{pmatrix}
B&0\\
0&\tran{B}{^{-1}}
\end{pmatrix}
\end{equation*}
generate $G_{1}=\Sp(W)$ where $A\in\Hom(W_{1}\du,W_{1})$ and $B\in\GL(W_{1})$. The character $\chi_{V}(\det B)$ is the quadratic character of $F^{\times}$; it is defined using the Hilbert symbol. The action of the unipotent group $U$ is dependent on the \textit{Gram matrix} of $x=\tran{(x_{1},x_{2})}{}\in V^{2}$ defined to be
\begin{equation*}
M_{x}=\big(\,(x_{i},x_{j})_{V}\,\big)_{ i,j}\,.
\end{equation*}
We define the character $\psi_{S}(u(A))=\psi(\Tr(SA))$. We also have that $\gamma_{4}\in\mu_{4}$ is a certain fourth root of unity and $\hat{\phi}$ is the Fourier transform of the Schwartz function $\phi$ (see \cite[\S 1]{roberts} for more details on this action). As in \cite{harris_kudla}, the extended action of $\weil$ to $Y$ is obtained by taking $(g,h)\in Y$, $\phi(x)\in\Sch(V^{2})$ and setting:
\begin{equation}\label{eq_weil_extn}
\weil(g,h)\,\phi(x)=\abs{\lambda(h)}_{F}^{-2}\,\weil(g_{1},1)\,\phi(h^{-1}x)
\end{equation}
where
\begin{equation*}
g_{1}=g\begin{pmatrix}
1_{2}&0\\
0&\lambda(g)^{-1}1_{2}
\end{pmatrix}\in G_{1}\,.
\end{equation*}

We now closely follow \cite[\S 5]{gan_ichino}. Define the induced Weil representation by compact induction:
\begin{equation*}
\Weil=\Ind_{R}^{H\times G^{+}}(\weil)\,.
\end{equation*}
If $\sigma$ is an irreducible, unitary, admissible representation of $H$ and $\bar{\sigma}$ is the conjugate representation of $\sigma$ then the maximal $\bar{\sigma}$-isotypic quotient of $\Weil$ is given by $\Weil\,/\cap\ker(\Psi)$ where $\Psi$ runs over $\Hom_{H}(\Weil,\bar{\sigma})$. This is a $\bar{\sigma}$-isotypic direct sum as an $H$-representation. Since $G^{+}$ naturally commutes with $H$ in $G^{+}\times H$, the space of $\Weil\,/\cap\ker(\Psi)$ inherits an action of $G^{+}$ and as a representation of $G^{+}\times H$ thus we may write

\begin{equation*}
\Weil\,/\cap\ker(\Psi)=\bar{\sigma}\boxtimes\Theta^{+}(\sigma)
\end{equation*}
where $\Theta^{+}(\sigma)$ is a smooth representation of $G^{+}$. We call $\Theta^{+}(\sigma)$ the \textit{big theta lift} of $\sigma$ to $G^{+}$. Whilst $\Theta^{+}(\sigma)$ may be zero, it is known that if this is not the case then $\Theta^{+}(\sigma)$ is of finite length, and hence is admissible, and has a unique, maximal, irreducible quotient \cite[Theorem A.1]{gan_ichino} which we denote $\theta^{+}(\sigma)$. This allow us to finally define the (\textit{local}) \textit{theta lift} of $\sigma$ to $G$ as

\begin{equation*}
\theta(\sigma)=\Ind_{G^{+}}^{G}(\theta^{+}(\sigma))\,.
\end{equation*}
By \cite[Lemma 5.2]{gan_ichino}, if $\sigma$ is non-zero and unitary\footnote{This is indeed the case when $\sigma$ is a local component of an irreducible, unitary, cuspidal automorphic representation of $H(\A)$ with a non-zero, cuspidal \textit{global} theta lift to $G(\A)$.} then $\theta(\sigma)$ is an irreducible representation of $G$. We obtain a unique (up to scalar) $Y$-equivariant, surjective map

\begin{equation}\label{eq_unique_map}
\theta\colon \V_{\sigma}\otimes\V_{\weil}\longrightarrow\V_{\theta(\sigma)}\,.
\end{equation}

\begin{rem}
That $\theta^{+}(\sigma)$ exists as a unique, maximal, irreducible representation is in fact the statement of the local Howe conjectures.
\end{rem}

\subsection{The global theta correspondence}\label{sec_theta_def}

In this section we return to our original notation where $F$ is a number field. The following construction follows \cite[\S 7.2]{gan_ichino}.

We have the fixed, non-trivial, additive character $\psi=\otimes_{v}\psi_{v}$ of $\A/F$ (chosen in \eqref{eq_std_char}). For each place $v$ of $F$ we let $\weil_{v}=\weil_{\psi_{v}}$ be the Weil representation of $Y(F_{v})$, with respect to $\psi_{v}$, realised in the Schwartz space $\V_{\weil_{v}}=\Sch(V^{2}(F_{v}))$. Let $\B_{\weil_{v}}\colon \V_{\weil_{v}}\otimes\,\V_{\bar{\weil}_{v}}\rightarrow \C$ be the canonical pairing defined by
\begin{equation*}
\B_{\weil_{v}}(\phi,\tilde{\phi})=\int_{V^{2}(F_{v})}\phi(x)\,\tilde{\phi}(x)\,dx\,.
\end{equation*}
The Weil representation of $Y(\A)$ is given by $\weil=\otimes_{v}\weil_{v}$, and comes equipped with the decomposable unitary pairing $\B_{\weil}=\prod_{v}\B_{\weil_{v}}$. The action of $\weil$ in $\V_{\weil}=\bigotimes_{v}\Sch(V^{2}(F_{v}))$ is applied place-by-place using the local action in \eqref{eq_schro_mod} and \eqref{eq_weil_extn}.

The global theta correspondence, in our setting, provides a cuspidal automorphic form on $G(\A)$ from one on $H(\A)$. We define this cusp form now. For a Schwartz function $\phi\in\V_{\weil}$ we note that the series
\begin{equation*}
\sum_{x\in V^{2}(F)}\weil(g,h)\phi(x)
\end{equation*}
is a smooth function on $(g,h)\in Y(F)\bk Y(\A)$ of moderate growth.

\begin{defn} Let $\sigma$ be an irreducible, cuspidal automorphic representation of $H(\A)$ and let $\phi\in\V_{\weil}$. Then for any $f\in\V_{\sigma}\subset\Ao(H)$ we define the \textit{theta integral}

\begin{equation}\label{eq_theta_int}
\theta(f,\phi;g)=\int_{H_{1}(F)\bk H_{1}(\A)}\,\sum_{x\in V^{2}(F)}\weil(g,hh_{g})\phi(x)\,f(hh_{g})\,dh
\end{equation}
where $h_{g}$ is any element in $H(\A)$ such that $\lambda(h_{g})=\lambda(g)$.
\end{defn}
 This integral is absolutely convergent and independent of the choice $h_{g}$ since all such elements are of the form $h_{g}h_{0}$ for $h_{0}\in H_{1}(\A)$. One computes the central character of $\theta(f,\phi)$ to be equal to $\omega_{\sigma}$, the central character of $f$ (since $\dim V=4$ is even).
 
By construction, $\theta(f,\phi)$ is a function on $G^{+}(F)\bk G^{+}(\A)$. By the natural inclusion of $G^{+}\emb G$ we extended $\theta(f,\phi)$ to a function on $G(F)\bk G(\A)$ by letting it take the value zero outside $G^{+}(\A)$. This extension is unique.

\begin{rem}
For any $h\in H=H^{0}\rtimes \mu_{2}$ there is an $h_{0}\in H^{0}$ with $\lambda(h)=\lambda(h_{0})$ since $h=h_{0}\varepsilon$ for $\varepsilon\in\mu_{2}\isom\langle\iota\rangle$ where $\iota\in H$ is the element defined in \eqref{eq_first_mention_iota} with $\lambda(\iota)=1$. Thus we may interchange $H$ with $H^{0}$ in the definition of $G^{+}$.
\end{rem}

\begin{defn}
Let $\theta(\sigma)$ be the automorphic representation of $G(\A)$ realised in the space
\begin{equation*}
\V_{\theta(\sigma)}=\left\lbrace\, R(g)\,\theta(f,\phi)\,:\,f\in\V_{\sigma},\,\phi\in\Sch(V^{2}(\A)),\,g\in G(\A)\, \right\rbrace.
\end{equation*}
We call $\theta(\sigma)$ the (\textit{global}) \textit{theta lift} of $\sigma$ to $G(\A)$.
\end{defn}

We shall fix assumptions on $\sigma$ (see Assumption \ref{assump_on_reps}) under which $\theta(\sigma)$ is \textit{cuspidal}. Under these conditions \cite[Lemma 7.12]{gan_ichino} applies so that $\V_{\theta(\sigma)}\neq 0$. We then obtain a $Y(\A)$-equivariant, surjective map
\begin{equation}\label{eq_theta_equi_map}
\theta\colon \V_{\sigma}\otimes\V_{\weil}\longrightarrow\V_{\theta(\sigma)}\,.
\end{equation}
We may restrict $\theta$ to $\V_{\weil_{v}}\otimes\V_{\sigma_{v}}$ at each place $v$ and conclude that, by the uniqueness of the local maps \eqref{eq_unique_map}, for $\sigma=\otimes_{v}\,\sigma_{v}$,
\begin{equation*}
\theta(\sigma)\isom\otimes_{v}\,\theta(\sigma_{v})
\end{equation*}
and is irreducible \cite[Lemma 7.2]{gan_ichino}. In particular, the local factors $\theta(\sigma_{v})$ are unitary and non-zero at each $v$.

\subsection{Automorphic induction}\label{sec_AI} An alternative description of the theta lift is that it arises due to a functorial transfer of representations from $H'(\A)$ to $\GSpf(\A)$ where $H'=\Res_{E/F}(\GLt)$ is the Weil restriction of scalars (meaning that $H'$ is unique in that $H'(F)=\GLt(E)$ as algebraic groups) and $E$ is a quadratic extension of $F$. For simplicity let us consider the trivial central character interpretation: the automorphic induction transfer between automorphic representations of the groups $H'_{1}=\Res_{E/F}\SLt$ and $G_{1}=\Spf$. On the one hand, the $L$-group of $G_{1}$ is $\dgroup{G}_{1}=\SO_{5}(\C)\times\Ga_{F}$ where $\Ga_{F}$ is the absolute Galois group of $F$. On the other hand, the $L$-group of $H'$ is
\begin{equation*}
\dgroup{H}'_{1}\isom \prod_{\Ga_{E}\bk\Ga_{F}}\SLt(\C)\rtimes \Ga_{F}
\end{equation*}
noting $\Ga_{E}\bk\Ga_{F}\isom\Gal(E/F)$ acts on the first factor in the product via permutations of the index set. Once again, make note of the isomorphism $\SO_{5}(\C)\isom\PGSpf(\C)$ which gives rise to an embedding
\begin{equation*}
\begin{array}{cccc}
\SLt(\C)\times\SLt(\C)\longrightarrow\Spf(\C)\,;&
\left(\begin{pmatrix}
a&b\\
c&d
\end{pmatrix},\begin{pmatrix}
a\pr&b\pr\\
c\pr&d\pr
\end{pmatrix}\right)&\longmapsto&\begin{pmatrix}
a&&b&\\
&a\pr&&b\pr\\
c&&d&\\
&c\pr&&d\pr
\end{pmatrix}
\end{array}
\end{equation*}
which in turn induces an $L$-homomorphism
\begin{equation*}
u\colon \dgroup{H}'_{1}\rightarrow\dgroup{G}_{1}\,.
\end{equation*}
On composing $u$ with a representation $r$ of the Weil--Deligne group $W'_{E}$ of $E$ into $\dgroup{H}'_{1}$ we obtain a representation $u\circ r$ that lands in $\dgroup{G}_{1}$. Noting $W'_{E}\subset W'_{F}$, this acquired representation is precisely the induced representation
\begin{equation*}
u\circ r=\Ind_{W'_{E}}^{W'_{F}}r
\end{equation*}
(on the Galois side). Whilst on the automorphic side we have an irreducible, cuspidal automorphic representation $\AI(\pi')$ of $G_{1}(\A)$ for each $\pi'$ on $H'_{1}(\A)=\SLt(\A_{E})$. A more general review in support of this exposition is given in \cite{cogdell}. 

A characteristic property of such a lift is that the $L$-function of the representations ($\AI(\pi')$ and $\pi'$) are equal, thus uniquely characterising the target $L$-packet. By the work of Roberts \cite[\S 8]{roberts} we find that this is also the case for the theta lift discussed in the previous two sections. Then, due to an exceptional isomorphism (see the next section, \S \ref{sec_fourdim}), we may realise the group $\GO_{2}$ as $\Res_{E/F}(\GLt)$ and hence any representation given by the above theta lift is functorial in this sense.

\section{Automorphic Representations of $\GO_{4}$}\label{sec_autm_reps_H}

To classify the image of the theta correspondence for $(\GO_{4},\GSpf)$ we provide a thorough review concerning the domain of the lift: we determine the structure of all four-dimensional quadratic spaces $V$, giving rise to $\GO(V)\isom\GO_{4}$, and with this analysis we examine the irreducible, cuspidal automorphic representations of $\GO(V)(\A)$. The review in this section is largely expository, however it includes new notation and crucial results which are used freely later on.

\subsection{Four-dimensional quadratic spaces and their similitude groups}\label{sec_fourdim}

Any four-dimensional quadratic space is isomorphic to a member of a family of spaces whose structure is explicit and indexed by two invariants: a quaternion algebra and a square-free integer (corresponding to the discriminant). For more details we refer to the exposition given in \cite[\S 2]{roberts}.

Consider a four-dimensional quadratic space $V$ over $F$ with $\disc(V)=e$. Let $E=K_{V}$ be the discriminant algebra of $V$ (defined in \eqref{eq_discalg}) and put $\Gal(E/F)=\left\lbrace 1, \kappa\right\rbrace$, using both $\kappa(z)$ and $z^{\kappa}$ to denote the image of $z\in E$ under $\kappa$. The usual norm and trace of $E/F$ are given by
\begin{equation*}
\N_{E/F}(z)=zz^{\kappa}\quad{\rm and}\quad\Tr_{E/F}(z)=z+z^{\kappa}\,.
\end{equation*}

\begin{defn}\label{def_X}
Let $B$ be an arbitrary $F$-algebra whose centre is $E$ with an involution $x\mapsto x\inv$ that fixes $E$. Call $B$ a \textit{quadratic-quaternion algebra} over $F$ if there is a quaternion algebra $D$, over $F$, contained in $B$ such that the natural map $D\otimes_{F} E\rightarrow B$, given by $x\otimes z\mapsto xz$, is an isomorphism of $E$-algebras and the canonical involution on $D$ is given by $x\mapsto x\inv$.
Choosing a $D$, there is no loss in generality in considering $B=D(E)$, the $E$-points of the $F$-algebra $D$. The \textit{norm} and \textit{trace} on $B$ are defined respectively as
\begin{equation*}
\N_{B}(x)=xx\inv\quad{\rm and}\quad\Tr_{B}(x)=x+x\inv\,.
\end{equation*}
When restricted to $D$ these are the usual reduced norm $\N_{D}$ and trace $\Tr_{D}$. Endow $B$ with the unique Galois action (with respect to $D$) by linearly extending the automorphism $\kappa$ of $E$ to $B$, that is $\kappa(xz)=x\kappa(z) $ for $z\in E$, $x\in D$. Denote this Galois action by $\kappa$ as well. Finally, define a second four-dimensional quadratic space (over $F$) by
\begin{equation*}
X=X_{D,e}=\lbrace\,x\in D(E)\,:\, \kappa(x)=x\inv\,\rbrace\,,
\end{equation*}
whose quadratic form, denoted $\N_{X}$, is given by the restriction of $\N_{B}$ to $X$. We find that this new space has $\disc X_{D,e}=\det \N_{X}=e$ upon computing the determinant of $N_{X}$.
\end{defn}

\begin{rem}
A \textit{Galois action} on $B$ is an $F$-automorphism $a\colon B\rightarrow B$ such that $a^{2}=1$ and $a(xz)=a(x){\color{red}\kappa}(z)$ for $z\in E$, $x\in B$. There is a bijection between Galois actions on $B$ and quaternion $F$-algebras contained in $B$.
\end{rem}

By \cite[Proposition 2.7]{roberts} we have the exact sequence
\begin{equation}\label{eq_GSO(X)_seq}
1\longrightarrow \Ex\arrup{\longrightarrow}{\Delta} \Fx\times\Bx\arrup{\longrightarrow}{\rho}\GSO(X)\longrightarrow 1\,,
\end{equation}
where the injection $\Delta\colon\Ex\rightarrow\Fx\times\Bx$ is given by $\Delta(z)=(\N_{E/F}(z),z)$ and the action of $\Fx\times\Bx$ on $X$ is given by
\begin{equation*}
\rho(s,a)x=s^{-1}a\, x\, \iota(a)\inv\,.
\end{equation*}
In particular, writing $\Delta\Ex$ for $\im(\Delta)$, we have
\begin{equation}\label{eq_GSO(X)_isom}
\Fx\times\Bx\,/\Delta\Ex\isom\GSO(X)\,.
\end{equation}
The similitude factor of an element $\rho(s,a)\in\GSO(X)$ is given by
\begin{equation*}
\lambda(\rho(s,a))=s^{-2}\N_{E/F}(\N_{B}(a))\,.
\end{equation*}

We denote by $\iota$ the restriction of the Galois action $\kappa$ to the subspace $X\subset B$ (again writing $\iota(x)$ and $x^{\iota}$ for the image of $x$ under $\iota$). The notation $\iota$ rightfully coincides with that already introduced in \S\ref{sec_quad_spaces} since the map $\iota$ is precisely the unique element of $\GO(X)$ satisfying the properties $\iota\in \Or(X)$, $\iota^{2}=1$ and $\det \iota = -1$ by \cite[Proposition 2.5 \& 2.7]{roberts}. We choose this element to fix, once and for all, the splitting
\begin{equation*}
\mu_{2}(F)\isom\langle\,\iota\,\rangle\quad{\rm and}\quad\GO(X)\isom\GSO(X)\rtimes\langle\,\iota\,\rangle\,.
\end{equation*}
Conjugating an element $\rho(s,a)\in\GSO(X)$ by $\iota$ gives the relation $\iota\rho(s,a)\iota=\rho(s,a^{\iota})$; we denote this \textit{adjoint} of $\iota$ action by
\begin{equation}\label{eq_adjoint_action}
\Ad(\iota)\colon \rho(s,a)\longmapsto \rho(s,a^{\iota})\,.
\end{equation}

\begin{prop}\label{prop_Xsuffices}
Let $V$ be an arbitrary four-dimensional quadratic space over $F$ of discriminant $e$. Then there exists a quaternion algebra $D$ over $F$ and an isomorphism $\gamma\colon V\longrightarrow X_{D,e}$ such that the map
\begin{equation*}
c_{\gamma}\colon \GSO(V)\longrightarrow\GSO(X_{D,e})\,,
\end{equation*}
given by $c_{\gamma}(g)= \gamma\circ g \circ \gamma^{-1}$, is an isomorphism of similitude groups. There is therefore no loss in generality in considering the space $\GSO(X_{D,e})$ in place of $\GSO(V)$
\end{prop}

\begin{proof}
See \cite[Proposition 2.8]{roberts}.
\end{proof}

From here on in, fix a quaternion algebra $D$ over $F$ and a square free integer $e$. We shall work with the four-dimensional quadratic space $X=X_{D,e}$. Fix notation for: the quadratic extension $E=F(\sqrt{e}\,)$ and the quadratic quaternion algebra $B=D(E)$. We assume the application of $V=X$ to the notations $H=\GO(V)$ etc.~of \S \ref{sec_groups_notation}.

\subsection{Local representation theory for $H(F_{v})$}

In this section let $v$ be a place of $F$ and suppress the subscript $v$ from the notation (for example, $F$ now denotes a local field). We shall systematically discuss the local (and later global) representation theory of $H$ in terms of that of $H^{0}$. We use this section to fix notation; this material has been previously considered in the expositions \cite[\S 1]{harris_soudry_taylor}, \cite[\S 2-4]{roberts} and \cite[\S A]{gan_ichino} -- we advise the reader to look there for details and proof. In \cite{takeda}, all restrictions in \cite{roberts} are removed, in particular the quadratic space $X$ may be of any signature.


\subsubsection{Admissible representations of $H^{0}$}\label{sec_admiss_repsH0}
In light of the isomorphism in \eqref{eq_GSO(X)_isom},
\begin{equation*}
\rho\colon\Fx\times\Bx\,/\Delta\Ex\arrup{\longrightarrow}{\sim} H^{0}\,,
\end{equation*}
let $(\tau,\V_{\tau})$ be an irreducible, admissible, unitary representation of $\Bx=\Bx(F)$ with central character $\omega_{\tau}$ (noting $Z_{\Bx}=\Ex$). Further assume that $\omega_{\tau}$ is $\Gal(E/F)$-invariant; thus we let $\nu$ be the unitary character of $\Fx$ such that
\begin{equation}\label{eq_cen_char_res1}
\omega_{\tau}=\nu^{-1}\circ\N_{E/F}\,.
\end{equation}
Every irreducible, admissible, unitary representation of $H^{0}$ may then be written in the form $\sigma_{0}=\sigma_{0}(\nu,\tau)$, for such a $\nu$ and $\tau$, by defining
\begin{equation*}
\sigma_{0}(\rho(s,a))=\nu(s)\tau(a)\,.
\end{equation*}
Both $\sigma_{0}$ and $\tau$ are realised in the same space $\V_{\sigma_{0}}=\V_{\tau}$. The requirement on $\nu$ \eqref{eq_cen_char_res1} ensures that $\sigma_{0}(\nu,\tau)$ is indeed trivial on $\Delta \Ex$. We identify the centre $Z_{H^{0}}\isom\Fx$, through $\rho$, as the set
\begin{equation*}
\left\lbrace\, (x^{-1},1)\,:\, x\in \Fx\right\rbrace\subset\ \Fx\times\Bx/\Delta \Ex,
\end{equation*}
from which we note that $\sigma_{0}$ has central character
\begin{equation*}
\omega_{\sigma_{0}}=\nu^{-1}\,.
\end{equation*}

\begin{defn}\label{def_distinguished} Suppose that $v$ is not split in $E$ (so that $E=E(F_{v})$ is a field). In this case, we
call an irreducible admissible representation $\sigma_{0}$ of $H^{0}$ \textit{distinguished} if
\[\sigma_{0}=\sigma_{0}(\omega_{\varrho}^{-1},\,\varrho_{E}^{D}\,)\,,\]
for some irreducible admissible representation $\varrho$ of $\GLt(F)$; denoting by $\varrho_{E}$ the base-change lift of $\varrho$ from $\GLt(F)$ to $\GLt(E)$, and appending the superscript $D$ to mean that $\varrho_{E}^{D}$ is the Jacquet--Langlands transfer of $\varrho_{E}$ from $\GL(E)$ to $\Dx(E)=\Bx$.
\end{defn}

The central character of such a distinguished $\sigma_{0}(\omega_{\varrho}^{-1},\,\varrho_{E}^{D}\,)$ is $\omega_{\varrho}$, the central character of $\varrho$. This follows from properties of the base-change lift (that $\omega_{\varrho_{E}}=\omega_{\varrho}\circ\N_{E/F}$). Distinguished representations are invariant under the adjoint action of $\iota$ on $H^{0}$ \eqref{eq_adjoint_action}. Hence a distinguished representation has the property that $\sigma_{0}\isom\sigma_{0}\circ\Ad(\iota)$ since we have $\varrho_{E}\isom\varrho_{E}\circ\iota$ (see \cite[\S 3]{arthur_clozel}).

\subsubsection{Admissible representations of $H$}\label{sec_admiss_reps_H}

To describe the irreducible, admissible representations of $H$ it suffices\footnote{Let $\sigma$ be an irreducible, admissible representation of $H$. Then either $\Res_{H^{0}}^{H}(\sigma)$ is irreducible, in which case $\sigma$ is an irreducible constituent of $\Ind_{H^{0}}^{H}(\Res_{H^{0}}^{H}(\sigma))$ and we are in the `invariant' case, or
\begin{equation*}
\Res_{H^{0}}^{H}(\sigma)=\sigma_{0,1}\oplus\sigma_{0,2}\,,
\end{equation*}
in which case $\sigma\isom\Ind_{H^{0}}^{H}(\sigma_{0,i})$ for either $i=1,2$; this is the `regular' case. Definition \ref{def_reg_inv} provides a full explanation of the invariant and regular cases.} to consider the induction of some $\sigma_{0}$ as $\sigma_{0}$ varies over the irreducible, admissible representations of $H^{0}$. To make this explicit, put $\sigma_{0}^{\iota}=\sigma_{0}\circ\Ad(\iota)$ and consider a second representation of $H^{0}$ in $\V_{\sigma_{0}}$ given by
\begin{equation*}
\sigma_{0}^{\iota}(h)v= \sigma_{0}(\iota h \iota)v\quad{\rm for}\,\,v\in\V_{\sigma_{0}}\,.
\end{equation*}
Now define the representation $(\hat{\sigma},\V_{\hat{\sigma}})$ of $H$ by setting $\V_{\hat{\sigma}}=\V_{\sigma_{0}}\oplus\V_{\sigma_{0}}$ and letting $H$ act on $u\oplus v$ by
\begin{equation*}
\left\lbrace\begin{array}{rcl}\vspace{0.05in}
\hat{\sigma}(h_{0}) u\oplus v  &=& \sigma_{0}(h_{0})u\oplus\sigma_{0}^{\iota}(h_{0})v\\
\hat{\sigma}(\iota) u\oplus v  &=& v\oplus u
\end{array}\right.
\end{equation*}
noting that any $h\in H$ may be written uniquely as $h=h_{0}\varepsilon$ for some $h_{0}\in H^{0}$ and $\varepsilon\in\mu_{2}$.

On the other hand, recall that $\Ind_{H^{0}}^{H}(\sigma_{0})$ is given by right translation in the space
\begin{equation*}
\left\lbrace\,f\colon H\longrightarrow\V_{\sigma_{0}}\,\mid\, f(h_{0}h)=\sigma_{0}(h_{0})f(h)\quad{\rm for}\,\,h_{0}\in H^{0}\,\right\rbrace.
\end{equation*}
One may check that there is an $H$-module isomorphism between the representations $\hat{\sigma}\isom\Ind_{H^{0}}^{H}(\sigma_{0})$. We will use $\hat{\sigma}$ as a model for $\Ind_{H^{0}}^{H}(\sigma_{0})$ from now on and proceed by dividing our analysis into two cases.

\begin{defn}\label{def_reg_inv} Let $\sigma_{0}$ be an irreducible, admissible representation of $H^{0}$.
\begin{itemize}\vspace{0.05in}
\item We say $\sigma_{0}$ is \textit{regular} if $\hat{\sigma}\isom\Ind_{H^{0}}^{H}(\sigma_{0})$ is irreducible. We find $\hat{\sigma}\isom\hat{\sigma}\otimes\sgn$ and, as $H^{0}$-representations, $\sigma_{0}\not\isom\sigma_{0}^{\iota}$. In this case denote $\sigma_{0}^{+}=\Ind_{H^{0}}^{H}(\sigma_{0})$.\vspace{0.05in}
\item We say $\sigma_{0}$ is \textit{invariant} if $\hat{\sigma}\isom\Ind_{H^{0}}^{H}(\sigma_{0})$ is reducible. We find $\hat{\sigma}\not\isom\hat{\sigma}\otimes\sgn$ and the adjoint action of $\iota$ in $\V_{\sigma_{0}}$ is trivial, that is, $\sigma_{0}\isom\sigma_{0}^{\iota}$. In this case
\begin{equation*}
\Ind_{H^{0}}^{H}(\sigma_{0})\isom\sigma_{0}^{+}\oplus\sigma_{0}^{-}
\end{equation*}
where $\sigma_{0}^{\pm}$ are two non-isomorphic irreducible representations of $H$.
\end{itemize}
\end{defn}

\begin{rem}
If $\sigma_{0}$ is distinguished then we have already noted that $\sigma_{0}$ is invariant. In this instance exactly one of $\sigma_{0}^{\pm}$ occurs in the theta correspondence with $\GSpf$ (see \cite[Theorem 3.4]{roberts}), denoting this representation by $\sigma_{0}^{+}$. Then for an irreducible, admissible representation $\sigma$ of $H$ we have that $\theta(\sigma)\neq 0$ if and only if $\sigma\not\isom\sigma_{0}^{-}$ for some distinguished $\sigma_{0}$ of $H^{0}$.
\end{rem}

\subsection{Global representation theory and automorphic forms for $H(\A)$}

In this section we reinstate $F$ as a number field. Our purpose is now to review the theory of automorphic forms on $H(\A)$. The following sources should be referred to for more detail: \cite[\S 1]{harris_soudry_taylor}, \cite[\S 5-7]{roberts} and \cite[\S 2]{gan_ichino}.

\subsubsection{Automorphic representations of $H^{0}(\A)$}\label{sec_autm_forms_H0}

The exactness of the sequence in \eqref{eq_GSO(X)_seq} (taking $F=F_{v}$ for each place $v$) implies that
\begin{equation*}
1\longrightarrow \Ax_{E}\arrup{\longrightarrow}{\Delta} \Ax\times\Bx(\A)\arrup{\longrightarrow}{\rho} H^{0}(\A)\longrightarrow 1
\end{equation*}
is also exact, where $\rho$ and $\Delta$ operate as in the exact sequence \eqref{eq_GSO(X)_seq} at each place. We identify $E(\A)$ with $\A_{E}$ and note $\Bx(\A)=\Dx(\A_{E})$. As subspaces of $L^{2}(\Ax\times\Bx(\A))$, the tensor product of the spaces of cusp forms $\Ao(\Fx)\otimes\Ao(\Bx)$ is dense in $\Ao(\Fx\times\Bx)$ and since these are spaces of smooth functions they are isomorphic. Any function on $\Ax\times\Bx(\A)\,/\Delta\Ax_{E}$ is a function on $\Ax\times\Bx(\A)$ subject to the constraint that it is constant on equivalence classes modulo $\Delta\Ax_{E}=\im(\Delta)$. In particular, if $\nu\colon \Fx\bk\Ax\rightarrow\Cx$ is a unitary Hecke character and $(\tau,\V_{\tau})$ an irreducible, cuspidal automorphic representation of $\Bx(\A)$ then, given some $\eta\in\V_{\tau}$, we have that $\nu\otimes\eta\in\Ao(\Fx\times\Bx\,/\Delta\Ex)$ if and only if

\begin{equation*}
\omega_{\tau}(z)=\nu^{-1}\circ\N_{E/F}(z)\qquad\forall\,\,z\in\Ax_{E}
\end{equation*}
where $\omega_{\tau}\colon \Ex\bk\Ax_{E}\rightarrow\Cx$ is the central character of $\tau$. Hence any irreducible, cuspidal automorphic representation of $H^{0}(\A)$ is of the form $\sigma_{0}=\sigma_{0}(\nu,\tau)$, for such a $\nu$ and $\tau$, where $\sigma_{0}$ is realised in the space of cusp forms $\V_{\sigma_{0}}=\left\lbrace\,\nu\otimes\eta\,\colon\,\eta\in\V_{\tau}\,\right\rbrace$ by the formula $\sigma_{0}(\rho(s,a))\,\nu\otimes\eta=\nu(s)\nu\otimes\tau(a)\eta$. Once again, the central character of $\sigma_{0}$ is $\omega_{\sigma_{0}}=\nu^{-1}$.


\subsubsection{Factorising automorphic representations of $\Bx(\A_{E})$ and $H^{0}(\A)$}\label{sec_factorise_reps} Consider the isomorphism
\begin{equation}\label{eq_rama}
 E\otimes_{F}F_{v}\isom \prod_{w\vert v} E_{w}\,,
\end{equation}
where the product is over all places of $w$ of $E$ above $v$ \cite[Proposition 4-40]{ramakrishnan}. One deduces
\begin{equation*}
\Bx(F_{v})\isom \prod_{w\vert v}\Bx(E_{w})\,.
\end{equation*}
Thus smooth representations of $\Bx(F_{v})$ are of the form $\tau_{v}=\otimes_{w\vert v}\,\tau_{w}$ where the $\tau_{w}$ are smooth representations of $\Bx(E_{w})$ for $w\vert v$. If $\sigma_{0}=\sigma_{0}(\nu,\tau)$ is an irreducible, cuspidal automorphic representation of $H^{0}(\A)$, as in \S \ref{sec_autm_forms_H0}, then by the tensor product theorem we may assume $\sigma_{0}\isom\otimes_{v}\sigma_{0,v}$ and $\nu=\otimes_{v}\nu_{v}$, over places $v$ of $F$, and $\tau\isom\otimes_{w}\tau_{w}$ over places $w$ of $E$. Then, by the previous remark, these local factors are related by $\sigma_{0,v}=\sigma_{0,v}(\nu_{v},\tau_{v})$ where $\tau_{v}=\otimes_{w\vert v}\,\tau_{w}$ and the space $\V_{\sigma_{0,v}}=\V_{\tau_{v}}=\otimes_{w\vert v}\V_{\tau_{w}}$ (as per \S \ref{sec_admiss_repsH0}).

\subsubsection{Automorphic representations of $H(\A)$}\label{sec_autm_forms_H}

\begin{assump}\label{assump_on_reps} Let $\sigma\isom\otimes_{v}\sigma_{v}$ be an irreducible, cuspidal automorphic representation of $H(\A)$ realised on the space $\V_{\sigma}\subset\Ao(H)$. For the remainder of this paper we shall assume the following for such a representation $\sigma$.
\begin{itemize}
\item[(1)] The Jacquet--Langlands transfer of $\sigma\vert_{\Bx(\A_{E})}$ to $\GLt(\A_{E})$ is cuspidal.\vspace{0.05in}
\item[(2)] There is at least one place $v$ for which $\sigma_{v}\isom\sigma_{v}\otimes\sgn$.\vspace{0.05in}
\item[(3)] If $\sigma_{v}\not\isom\sigma_{v}\otimes\sgn$ then $\sigma_{v}\not\isom\sigma_{0,v}^{-}$ for any distinguished (and invariant) admissible representation $\sigma_{0,v}$ of $H^{0}(F_{v})$.
\end{itemize}
These conditions are imposed in \cite{gan_ichino}, thus ensuring that $\theta(\sigma)$ is both cuspidal (1) and non-zero (3). Condition (2) is necessary to compute the Petersson inner product of the theta lift $\theta(\sigma)$ in \eqref{eq_GI_global_theta_pair}.
\end{assump}
We now determine all such $\sigma$ by considering their restriction to $H^{0}(\A)$. (This top-down approach contrasts with the bottom-up analysis used in the local setting.) To this end, define a (possibly infinite) subset of the places of $F$ by
\begin{equation*}
\Set=\left\lbrace\,v\,:\,\sigma_{v}\isom\sigma_{v}\otimes\sgn\,\right\rbrace.
\end{equation*}
Assumption \ref{assump_on_reps}-(2) implies $\Set\neq\emptyset$. By the tensor product theorem, fix an isomorphism of $H(\A)$-representations

\begin{equation*}
\begin{array}{l}
{\displaystyle\V_{\sigma}\isom\bigotimes_{v}}\pr \displaystyle\V_{\sigma_{v}}=\lim_{\substack{\longrightarrow\\S}}\big(\bigotimes_{v\in S}\V_{\sigma_{v}}\big)\otimes\big(\bigotimes_{v\not\in S}f^{\circ}_{v}\big)
\end{array}
\end{equation*}
where $\V_{\sigma_{v}}$ is the space of $\sigma_{v}$ and, for a sufficiently large set of places $S$ outside which $\sigma_{v}$ is unramified, $f^{\circ}_{v}\in\V_{\sigma_{v}}$ is an $H(\mathcal{O}_{v})$-invariant (spherical) vector for $v\not\in S$. By analogy with our local discussion \S \ref{sec_admiss_reps_H}, the restriction of $\sigma_{v}$ to $H^{0}_{v}$ gives rise to two cases.

\begin{itemize}\vspace{0.05in}

\item If $v\in\Set$ then $\sigma_{v}\vert_{H^{0}_{v}}\isom\sigma_{0,v}\oplus\sigma_{0,v}^{\iota}$ where $\sigma_{0,v}$ is an irreducible representation of $H^{0}_{v}$ with $\sigma_{0,v}\not\isom\sigma_{0,v}^{\iota}$. Earlier, we called such a $\sigma_{0,v}$ \textit{regular} and noted that its induction, $\hat{\sigma}_{v}$, was irreducible. The space of $\sigma_{v}$ decomposes as $\V_{\sigma_{v}}=\V_{\sigma_{0,v}}\oplus\V_{\sigma_{0,v}^{\iota}}$, realising the space $\V_{\sigma^{\iota}_{0,v}}\isom\sigma_{v}(\iota)\V_{\sigma_{0,v}}$. For almost all $v\in\Set$, the spherical vector $f^{\circ}_{v}= \mathfrak{F}^{\circ}_{v}+\sigma_{v}(\iota)\mathfrak{F}^{\circ}_{v}\in \V_{\sigma_{0,v}}\oplus\V_{\sigma_{0,v}^{\iota}}$ where $\mathfrak{F}^{\circ}_{v}$ is an $H^{0}(\mathcal{O}_{v})$-invariant vector.\vspace{0.1in}

\item If $v\not\in\Set$ then $\sigma_{v}\vert_{H^{0}_{v}}$ is irreducible and \textit{invariant}; we have $\V_{\sigma_{v}}=\V_{\sigma_{0,v}}$ and the spherical vector $f^{0}_{v}=\mathfrak{F}^{\circ}_{v}$ is $H^{0}(\mathcal{O}_{v})$-invariant. Write $\sigma_{0,v}=\sigma_{v}\vert_{H^{0}_{v}}$ in this case.\vspace{0.05in}
\end{itemize}
Let $S$ be a sufficiently large set of places of $F$ and put $S\pr=S\smallsetminus(S \cap\Set)$. For each $\varepsilon=(\varepsilon_{v})\in\mu_{2}(F_{S\cap\Set})$ define $\V^{\varepsilon}_{\sigma,S}\subset\V_{\sigma}$ by
\begin{equation*}
\V^{\varepsilon}_{\sigma,S}\isom\big(\bigotimes_{v\in S\cap\Set}\sigma_{v}(\varepsilon_{v})\V_{\sigma_{0,v}} \big)\otimes\big(\bigotimes_{v\in S\pr}\V_{\sigma_{0,v}}\big)\otimes\big(\bigotimes_{v\not\in S}f^{\circ}_{v}\big)\,.
\end{equation*}

Viewing $\sigma$ from a different perspective, consider the space of restricted functions
\begin{equation*}
\V_{\sigma}\vert_{H^{0}(\A)}=\left\lbrace\, f\vert_{H^{0}(\A)}\,:\,f\in\V_{\sigma}\,\right\rbrace\,.
\end{equation*}
By \cite[Lemma 2]{harris_soudry_taylor} there exists an irreducible, cuspidal automorphic representation $\sigma_{0}$ of $H^{0}(\A)$ realised in a space of cusp forms $\V_{\sigma_{0}}$ such that
\begin{equation}\label{eq_hst_Ltwo}
\V_{\sigma}\vert_{H^{0}(\A)}=\V_{\sigma_{0}}\oplus\V_{\sigma_{0}}^{\iota}\,,
\end{equation}
defining $\V_{\sigma_{0}}^{\iota}=\left\lbrace\, f^{\iota}=f\circ\Ad(\iota)\,:\,f\in\V_{\sigma_{0}}\,\right\rbrace$, and such that $\sigma_{0}\not\isom\sigma_{0}^{\iota}$. In this circumstance we shall say $\sigma$ \textit{lies above} $\sigma_{0}$.

Applying the tensor product theorem and comparing the local components of $\sigma$ and $\sigma_{0}$ with those $\sigma_{0,v}$ already defined, we may assume that $\sigma_{0}\isom\otimes_{v}\sigma_{0,v}$. Moreover, choosing $\varepsilon=1$, the restriction of the space of functions $\V^{1}_{\sigma,S}=\V_{\sigma_{0}}$.

As a final remark, \eqref{eq_hst_Ltwo} shows that $\V_{\sigma,S}^{\varepsilon}\vert_{H^{0}(\A)}=\left\lbrace 0 \right\rbrace$ unless $\varepsilon\in\mu_{2}(F)$ (else contradicting that $\sigma_{0}\not\isom\sigma_{0}^{\iota}$). In particular, consider evaluating a function $f\in \V^{1}_{\sigma,S}$ on
\begin{equation*}
H(\A)=\bigcup_{\varepsilon\in\mu_{2}(F_{S\cap\Set})} H^{0}(\A)\mu_{2}(\A^{S\cap\Set})\varepsilon\,.
\end{equation*}
For $\varepsilon\in\mu_{2}(F_{S\cap\Set})$ we have $\V^{\varepsilon}_{\sigma,S}=\sigma(\varepsilon)\V^{1}_{\sigma,S}$ and hence $\sigma(\varepsilon)f=0$ unless $\varepsilon\in\mu_{2}(F)$. We then obtain \cite[Lemma 2.2]{gan_ichino}:
\begin{equation}\label{eq_gi_lem2.2}
\supp(f)\subset H^{0}(\A)\mu_{2}(\A^{S\cap\Set})\,\cup\, H^{0}(\A)\mu_{2}(\A^{S\cap\Set})\iota\,.
\end{equation}

\subsection{Explicit unitary pairings and the Petersson inner product}\label{sec_pairings}

The unique (up to scalar) unitary pairings $\B_{\sigma_{0,v}}\colon\V_{\sigma_{0,v}}\otimes\V_{\bar{\sigma}_{0,v}}\rightarrow\C$ associated to the local components $\sigma_{0,v}=\sigma_{0,v}(\nu_{v},\tau_{v})$ of $\sigma_{0}(\nu,\tau)$, as in \S \ref{sec_factorise_reps}, are precisely the pairings on $\V_{\tau_{v}}\otimes\V_{\bar{\tau}_{v}}$ since $\V_{\sigma_{0,v}}=\V_{\tau_{v}}$ and $\nu_{v}$ is unitary.

We therefrom assume that, whenever $\B_{\tau_{v}}$ is specified, by $\B_{\sigma_{0,v}}$ we always mean the pairing $\B_{\sigma_{0,v}}=\B_{\tau_{v}}$. The possible splitting of $v$ in $E$ must also be accounted for in our choice of pairing: we make the convention that if $\B_{\tau_{w}}$ is a specified pairing on $\V_{\tau_{w}}$ (for each place $w$ of $E$ lying above $v$) then
\begin{equation*}
\B_{\tau_{v}}=\otimes_{w\vert v}\,\B_{\tau_{w}}\vspace{0.04in}
\end{equation*}
is the fixed pairing on $(\otimes_{w\vert v}\,\V_{\tau_{w}})\otimes(\otimes_{w\vert v}\,\V_{\bar{\tau}_{w}})$ and hence also on $\V_{\sigma_{0,v}}\otimes\V_{\bar{\sigma}_{0,v}}$.

If $\V_{\sigma_{0,v}}$ carries a pairing $\B_{\sigma_{0,v}}$ and $\sigma_{v}$ is an irreducible, admissible representation above $\sigma_{0,v}$ then we choose to consider a specific pairing on $\V_{\sigma_{v}}$:

\begin{itemize}\vspace{0.05in}
\item If $v\in\Set$ then $\V_{\sigma_{v}}\vert_{H^{0}_{v}}=\V_{\sigma_{0,v}}\oplus\V_{\sigma_{0,v}^{\iota}}$ is irreducible; take the pairing
\begin{equation*}
\B_{\sigma_{v}}\colon (\V_{\sigma_{0,v}}\oplus\V_{\sigma_{0,v}^{\iota}})\otimes (\V_{\bar{\sigma}_{0,v}}\oplus\V_{\bar{\sigma}_{0,v}^{\iota}})\longrightarrow\C
\end{equation*}
given by $\B_{\sigma_{v}}((x +\sigma_{v}(\iota) y),\,(\tilde{x} +\bar{\sigma}_{v}(\iota) \tilde{y}))=\frac{1}{2}(\B_{\sigma_{0,v}}(x,\tilde{x})+\B_{\sigma_{0,v}}(y,\tilde{y}))$.\vspace{0.1in}

\item If $v\not\in\Set$ then $\V_{\sigma_{v}}\vert_{H^{0}_{v}}=\V_{\sigma_{0,v}}$ is irreducible; take $\B_{\sigma_{v}}=\B_{\sigma_{0,v}}$.\vspace{0.05in}
\end{itemize}
This pairing is chosen carefully so that we may factorise the Petersson inner products $\B_{\sigma}$ and $\B_{\sigma_{0}}$ when $\sigma=\otimes_{v}\sigma_{v}$ is an automorphic representation of $H(\A)$ that lies above $\sigma_{0}=\otimes_{v}\sigma_{0,v}$. As before, fix an isomorphism for the conjugate representation $\bar{\sigma}\isom\otimes_{v}\bar{\sigma}_{v}$
\begin{equation*}
\V_{\bar{\sigma}}\isom\bigotimes_{v}\V_{\bar{\sigma}_{v}}=\lim_{\substack{\longrightarrow\\S}}\big(\bigotimes_{v\in S}\V_{\bar{\sigma}_{v}}\big)\otimes\big(\bigotimes_{v\not\in S}\tilde{f}^{\circ}_{v}\big)
\end{equation*}
where, for a sufficiently large set of places $S$ outside which $\bar{\sigma}_{v}$ is unramified, $\tilde{f}^{\circ}_{v}\in\V_{\sigma_{v}}$ is an $H(\mathcal{O}_{v})$-invariant (spherical) vector for $v\not\in S$. If $v\in\Set$ then $\tilde{f}^{\circ}_{v}= \widetilde{\mathfrak{F}}^{\circ}_{v}+\sigma_{v}(\iota)\widetilde{\mathfrak{F}}^{\circ}_{v}$ where $\widetilde{\mathfrak{F}}^{\circ}_{v}$ is an $H^{0}(\mathcal{O}_{v})$-invariant vector and if $v\not\in\Set$ then $\tilde{f}^{0}_{v}=\widetilde{\mathfrak{F}}^{\circ}_{v}$ is $H^{0}(\mathcal{O}_{v})$-invariant.

\begin{lem}\label{lem_pet_factor_assump}
For almost all $v$ suppose that $\B_{\sigma_{0,v}}$ is normalised by $\B_{\sigma_{0,v}}(\mathfrak{F}^{\circ}_{v},\widetilde{\mathfrak{F}}^{\circ}_{v})=1$. Then, if the pairings $\B_{\sigma_{0,v}}$ are normalised so that the Petersson inner product may be factorised as $\B_{\sigma_{0}} = \prod_{v} \B_{\sigma_{0,v}}$, we additionally have the following decomposition:
\begin{equation*}
\B_{\sigma} = \prod_{v} \B_{\sigma_{v}}\,.
\end{equation*}
\end{lem}

\begin{proof}
See \cite[Lemma 2.3]{gan_ichino}.
\end{proof}

The Petersson inner products for both the automorphic representations $\sigma_{0}=\sigma_{0}(\nu,\tau)$ and $\tau$ agree: if we have $(\eta,\tilde{\eta})\in\V_{\tau}\otimes\V_{\tau}$ and $f_{0}=\nu\otimes\eta$, $\tilde{f}_{0}=\bar{\nu}\otimes\tilde{\eta}$ then
\begin{equation*}
\B_{\sigma_{0}}(f_{0},\,\tilde{f}_{0})=\B_{\tau}(\eta,\tilde{\eta})\,.
\end{equation*}

The Petersson inner product associated to the unitary Hecke character $\chi$ of $\Ax_{K}$ (trivially) coincides with the Tamagawa number of $\Fx\bk\Kx$, given by $\Vol(\Ax\Kx\bk\Ax_{K})=2$ (see \cite[p.~44]{liu}). Underlying our calculations we choose local pairings $\B_{\chi_{v}}=1$ at all $v$.

\subsection{The Petersson inner product for theta lifts}\label{sec_theta_pairings}

Gan--Ichino prove a decomposition of the Petersson inner product for the theta lift $\theta(\sigma)$ with respect to some specified pairings for the local factors $\theta(\sigma_{v})$. This result assumes that $F$ is a totally real number field and that $\sigma=\otimes_{v}\sigma_{v}$ is an irreducible, cuspidal automorphic representation of $H(\A)$ satisfying Assumption \ref{assump_on_reps}. In particular, in this assumption, conditions (2) and  (3) are used explicitly in the proof of this formula whereas the totally real assumption is required for an application of the Siegel--Weil formula.

Fix a choice of local pairings $\B_{\sigma_{0,v}}$ such that $\B_{\sigma_{0}}=\prod_{v}\B_{\sigma_{0,v}}$ and consider the pairings $\B_{\sigma_{v}}$, defined in \S \ref{sec_pairings}. For $(f,\tilde{f})\in\V_{\sigma}\otimes\V_{\bar{\sigma}}$ and Schwartz functions $(\phi_{v},\tilde{\phi}_{v})\in\V_{\weil_{v}}\otimes\V_{\bar{\weil}_{v}}$ define
\begin{equation}\label{eq_local_theta_pair}
\begin{array}{l}\vspace{0.14in}
\B_{\theta(\sigma_{v})}(\theta(f_{v},\phi_{v}),\theta(\tilde{f}_{v},\tilde{\phi}_{v}))=\hfill\\
\hspace{0.6in}\displaystyle\dfrac{\zeta_{F_{v}}(2)\,\zeta_{F_{v}}(4)}{L(\std,\sigma_{v},1)}\,\int_{H_{1}(F_{v})}\B_{\weil_{v}}(\weil_{v}(h_{v})\phi_{v},\tilde{\phi}_{v})\,\B_{\sigma_{v}} (\sigma_{v}(h_{v})f_{v},\tilde{f}_{v})\, dh_{v}
\end{array}
\end{equation}
where the Haar measures $dh_{v}$ on $H_{1,v}$ are those determined by a differential form (of top degree) on $H_{1}$ and the self-dual Haar measure on $\Fx_{v}$ (with respect to $\psi_{v}$) -- these in fact give the Tamagawa measure $dh=\prod_{v} dh_{v}$ of $H_{1}(\A)$ (as constructed in \cite{weil_tam}).

Gan--Ichino take care in deriving the constant of proportionality between the Petersson inner product for $\theta(\sigma)$ and $\prod_{v}\B_{\theta(\sigma_{v})}$. With Assumption \ref{assump_on_reps} we have \cite[Proposition 7.13]{gan_ichino}:

\begin{equation}\label{eq_GI_global_theta_pair}
\B_{\theta(\sigma)}=\dfrac{L(\std,\sigma,1)}{\zeta_{F}(2)\,\zeta_{F}(4)}\,\prod_{v}\B_{\theta(\sigma_{v})}\,.
\end{equation}

\section{Global Calculation: The Bessel Period}\label{sec_global}

Preliminary remarks aside, we use this section to determine the form of the Bessel period \eqref{eq_bessel_def} for the theta integral \eqref{eq_theta_int}. First of all we explicitly highlight any running assumptions and notations (in addition to those in Assumption \ref{assump_on_reps}).

\subsection{Hypotheses and variables}\label{sec_hypotheses}

We have fixed the (base) number field $F$ to be totally real. This assumption permits the use of the Siegel--Weil formula (or rather its corollary; the Rallis inner product formula) in a calculation made in \cite{gan_ichino} whereby the Petersson inner product for a theta lift is computed in terms of local pairings (see Proposition \ref{eq_GI_global_theta_pair}).

In \S \ref{sec_fourdim} we acquired the following notation an assumptions: $V$ is a four-dimensional quadratic space (over $F$) of discriminant $\disc V=e$; we assume that $e$ is not a square in $\Fx$ (since the case when $e$ is a square has been settled by Liu); Proposition \ref{prop_Xsuffices} implies that it suffices to fix such an $e\in\Fx$ and a (possibly split) quaternion algebra $D$ over $F$ and consider instead the space $X=X_{D,e}$ -- we do this and apply $V=X$ to the notations $H=\GO(V)$ etc.~of \S \ref{sec_groups_notation}; fix once and for all $E=F(\sqrt{e}\,)$ and $B=D(E)\isom D\otimes_{F}E$.

Our result is concerned with irreducible, cuspidal automorphic representations of $\GSpf(\A)$ lifted from $\GO(V)(\A)$ by the theta correspondence (\S \ref{sec_theta_section}).

\begin{assump}\label{assump_triv_cenchar}
We only consider representations of $\PGSpf(\A)\isom\SO_{5}(\A)$; these are precisely the representations of $\GSpf(\A)$ with \textit{trivial central character}.
\end{assump}

Note that the theta lift $\theta(\sigma)$ has central character $\omega_{\theta(\sigma)}=\omega_{\sigma}$ so we assume $\omega_{\sigma}=1$. If $\sigma$ lies above $\sigma_{0}=\sigma_{0}(\nu,\tau)$, as in \eqref{eq_hst_Ltwo}, then $\nu=\omega_{\sigma}^{-1}=1$. For the remainder of this paper, we keep in mind a fixed irreducible, cuspidal automorphic representation $\sigma\isom\otimes_{v}\sigma_{v}$ of $H(\A)$ (in the space $\V_{\sigma}$) lying above $\sigma_{0}=\sigma_{0}(1,\tau)$ where $\tau\isom\otimes_{w}\tau_{w}$ is an irreducible, cuspidal automorphic representation of $\Bx(\A)$ whose central character $\omega_{\tau}=1$. Also fix a factorisation for the conjugate representation $\bar{\sigma}\isom\otimes_{v}\bar{\sigma}_{v}$. There exists a set of places $\Set=\left\lbrace\,v\,:\,\sigma_{v}\isom\sigma_{v}\otimes\sgn\,\right\rbrace$ which determine $\sigma$ uniquely given $\sigma_{0}$ (see \S \ref{sec_autm_forms_H}).

Let $f=\otimes_{v}f_{v}\in\V_{\sigma}$ be a pure tensor, fixing this choice throughout the remainder of this paper. We identify a factorisation for the conjugate of $f$ by 
\begin{equation}\label{eq_fix_conjugate_factors}
\bar{f}=\otimes_{v}\bar{f}_{v}
\end{equation}
so that it makes sense to talk about a specific $\bar{f}_{v}$ corresponding to a local factor $f_{v}$ of $f$. Similarly, we fix factorisations for the Schwartz functions $\phi=\otimes_{v}\phi_{v}\in\V_{\weil}$ and $\bar{\phi}=\otimes_{v}\bar{\phi}_{v}\in\V_{\bar{\weil}}$. 

Choose a series of local unitary pairings $\B_{\tau_{w}}$ on $\V_{\tau_{w}}\otimes\V_{\bar{\tau}_{w}}$, for each place $w$ of $E$, such that the Petersson pairing has the factorisation $\B_{\tau}=\prod_{w}\B_{\tau_{w}}$. Due to the choices of \S \ref{sec_pairings}, we then automatically obtain the pairings $\B_{\sigma_{0,v}}$ and $\B_{\sigma_{v}}$ for $\sigma_{0,v}$ and $\sigma_{v}$, respectively. Note that these depend on the place $v$ of $F$. The Petersson pairings will satisfy a similar factorisation
\begin{equation}\label{eq_choose_sigma_pairing}
\B_{\sigma_{0}}=\prod_{v}\B_{\sigma_{0,v}}\quad {\rm and}\quad \B_{\sigma}=\prod_{v}\B_{\sigma_{v}}\,.
\end{equation}

Fix another non-square element $d\in\Fx$. Let $K=F(\sqrt{d}\,)$ and define a Hecke character $\chi\colon\Kx\bk\Ax_{K}\rightarrow\Cx$. Then $K$ and $\chi$ index a unique Bessel period (see \S \ref{sec_bessel_period}). We impose the following assumption, which is essentially the trivial central character assumption when considering $\chi$ as a representation of $\GSO(X)$.

\begin{assump}\label{assump_on_K_chi}
Suppose that $\chi$ is unitary and satisfies $\chi\vert_{\Ax_{F}}=1$.
\end{assump}

\subsection{Explicit vectors}\label{sec_explicit_vectors}
We shall consider vectors $\varphi=\theta(f,\phi)$ for $f\in\V_{\sigma}$ and $\phi\in\V_{\weil}$ such that $\varphi=\otimes_{v}\varphi_{v}$ is a pure tensor. The global map $\theta$ of \eqref{eq_theta_equi_map} is linear in each variable and hence $\varphi$ is a pure tensor when both $f=\otimes_{v}f_{v}$ and $\phi=\otimes_{v}\phi_{v}$ are pure tensors (as we have assumed). We fix the notation $\varphi_{v}=\theta(f_{v},\phi_{v})$ for the local components in the factorisation of $\theta(f,\phi)$ (noting that this is necessary as each local map $\theta$ \eqref{eq_unique_map} is only unique up to a scalar constant).

Our choice of local vectors $\bar{f}_{v}\in\V_{\bar{\sigma}_{v}}$ (see \eqref{eq_fix_conjugate_factors}) and $\bar{\phi}_{v}\in\V_{\bar{\weil}_{v}}$ give rise to the factors in $\bar{\varphi}=\otimes_{v}\bar{\varphi}_{v}$ in the sense that
\begin{equation}\label{eq_fix_theta_conj_factors}
\bar{\varphi}_{v}=\overline{\theta(f_{v},\phi_{v})}=\theta(\bar{f}_{v},\bar{\phi}_{v})
\end{equation}
by the uniqueness of \eqref{eq_unique_map} and \eqref{eq_theta_equi_map}, the choice of vectors $\varphi_{v}=\theta(f_{v},\phi_{v})$ and then applying \cite[Proposition 5.5]{gan_ichino}.

\begin{lem}\label{lem_trick_with_iota}
Define $f^{\iota}(h)=f(h\iota)$. One has $\theta(f,\phi)=\theta(f^{\iota},\weil(\iota)\phi)$. 
\end{lem}

\begin{proof}
We compute
\begin{equation*}
\begin{array}{rcl}
\theta(f^{\iota},\weil(\iota)\phi;g)&=&\displaystyle\int_{H_{1}(F)\bk H_{1}(\A)}\,\sum_{x\in X^{2}(F)}\weil(g,hh_{g}\iota)\phi(x)\,f(\iota hh_{g}\iota)\,dh\,.\\\vspace{0.07in}
&=&\displaystyle\int_{H_{1}(F)\bk H_{1}(\A)}\,\sum_{\iota x\in X^{2}(F)}\weil(g,\iota h\iota h_{g}\pr)\phi(\iota x)\,f(\iota h\iota h_{g}\pr)\,dh\\
&=&\theta(f,\phi;g)
\end{array}
\end{equation*}
where $h_{g}\pr=\iota h_{g}\iota$ has $\lambda(h_{g}\pr)=\lambda(g)$. Here we use the automorphy of $f$ under $\iota\in\mu_{2}(F)$ and rearrange the summation by $x\mapsto \iota x$. The Tamagawa measure $dh$ is invariant under the transformation $h\mapsto \iota h\iota$.
\end{proof}

Since an arbitrary element $f$ of $\V_{\sigma}$ is of the form $f=f_{1}+f_{2}^{\iota}$ for some $f_{1},f_{2}\in\V_{\sigma,S}^{1}$ (by \eqref{eq_hst_Ltwo}), Lemma \ref{lem_trick_with_iota} implies
\begin{equation*}
\theta(f_{1}+f_{2}^{\iota}, \phi)=\theta(f_{1},\phi)+\theta(f_{2},\weil(\iota)\phi)\,.
\end{equation*}
There is then no loss in generality in restricting our choice of $f\in\V_{\sigma}$ to the following.

\begin{assump}
For a fixed, finite set $S$, assume that $f=\otimes_{v}f_{v}\in\V_{\sigma,S}^{1}$ is a pure tensor. Such an $f$ satisfies the property that $f\vert_{H^{0}(\A)}\in\V_{\sigma_{0,v}}$.
\end{assump}

Recalling that $\tau$ is the automorphic representation of $\Bx(\A)$ such that $\sigma_{0}=\sigma_{0}(1,\tau)$, we denote by
\begin{equation*}
\eta=\otimes_{w}\eta_{w}\in\V_{\tau}
\end{equation*}
(decomposed over places $w$ of $E$) the function such that
\begin{equation*}
f(\rho(s,a))=\eta(a)\,.
\end{equation*}
The local factors of these functions are identified by $f_{v}=\otimes_{w\vert v}\eta_{w}$ (see \S \ref{sec_factorise_reps}). Note that $f^{\iota}=\sigma(\iota)f$, and since $\theta(f,\phi)=\otimes_{v}\theta(f_{v},\phi_{v})$, Lemma \ref{lem_trick_with_iota} implies that for each $v$
\begin{equation*}
\theta(f_{v},\phi_{v})=\theta(\sigma_{v}(\iota)f_{v},\weil_{v}(\iota)\phi_{v})\,.
\end{equation*}

\subsection{A calculation in terms of the variant theta integral}\label{sec_cal_in_var}

To simplify matters (overall) we introduce the \textit{variant} theta integral (to be compared with \eqref{eq_theta_int}):

\begin{equation*}
\theta^{0}(f,\phi;g)=\int_{H^{0}_{1}(F)\bk H^{0}_{1}(\A)}\,\sum_{x\in X^{2}(F)}\weil(g,h_{0}h_{g})\phi(x)\,f(h_{0}h_{g})\,dh_{0}
\end{equation*}
where the domain is defined in terms of the connected, index-two subgroup $H^{0}_{1}$ of $H_{1}$. For this function we also have
\begin{equation}\label{eq_theta0_inv_iota}
\theta^{0}(f,\phi)=\theta^{0}(f^{\iota},\weil(\iota)\phi)
\end{equation}
by a computation identical to Lemma \ref{lem_trick_with_iota}. Observe how $\theta^{0}(f,\phi)$ is related to $\theta(f,\phi)$.

\begin{lem}\label{lem_int_func}
For any integrable function $\Phi$ on $H_{1}(F)\bk H_{1}(\A)$ we have
\[\displaystyle\int_{H_{1}(F)\bk H_{1}(\A)}\Phi(h)\,dh=\displaystyle\int_{\mu_{2}(F)\bk\mu_{2}(\A)}\int_{H^{0}_{1}(F)\bk H^{0}_{1}(\A)}\,\Phi(h_{0}\varepsilon)\,dh_{0}\,d\varepsilon\]
where $d\varepsilon$ is the Tamagawa measure on $\mu_{2}(\A)$.
\end{lem}

Since $\theta^{0}(f,\phi;g)$ is independent of a particular choice of $h_{g}$ we may apply Lemma \ref{lem_int_func} and substitute $h_{g}\mapsto\varepsilon h_{g}\varepsilon$ (as $\lambda(\varepsilon)=1$) to find
\begin{equation}\label{eq_variant_formula}
\theta(f,\,\phi;g)=\displaystyle\int_{\mu_{2}(F)\bk\mu_{2}(\A)}\theta^{0}(\sigma(\varepsilon) f,\,\weil(\varepsilon)\phi;g)\,d\varepsilon\,.
\end{equation}
This relation permits one to consider the refined quantity $\Pe(\theta^{0}(f,\phi),\chi)$.

\subsection{Unfolding the Weil representation}\label{sec_unfolding} By definition (see \eqref{eq_bessel_def}) we have

\begin{equation*}\label{eq_period1}
\Pe(\theta^{0}(f,\phi),\chi)=\displaystyle\int_{\Ax T(F)\bk T(\A)}\int_{U(F)\bk U(\A)}\theta^{0}(f,\phi;u\hat{g})\,\chi(g)\,\psi^{-1}_{S}(u)\,du\,dg
\end{equation*}
so we start out by computing

\begin{equation*}
\theta^{0}(f,\phi;u\hat{g})=\int_{H^{0}_{1}(F)\bk H^{0}_{1}(\A)}\,\sum_{x\in X^{2}(F)}\weil(u\hat{g},h_{0}h_{g})\phi(x)\,f(h_{0}h_{g})\,dh_{0}\,.
\end{equation*}
Applying the action of $\weil$ to $\phi=\otimes_{v}\phi_{v}$ (place-by-place)  we find that

\begin{equation*}
\begin{array}{rcl}\vspace{0.1in}
\displaystyle\weil(u\hat{g},h_{0}h_{g})\,\phi(x)&=&\bigg(\displaystyle\prod_{v}\,\chi_{V,v}(\det(g_{v}))\,\abs{\lambda(g_{v})}_{v}^{-2}\,\abs{\det g_{v}}^{2}_{v}\bigg)\,\psi_{M_{x}}(u)\,\phi(h_{g}^{-1}h_{0}^{-1}xg)\\
&=&\psi_{M_{x}}(u)\,\phi(h_{g}^{-1}h_{0}^{-1}xg)\,,
\end{array}
\end{equation*}
recalling $\prod_{v}\,\chi_{V,v}(\det(g_{v}))=1$ (by quadratic reciprocity) and $\psi_{M_{x}}$ is the character of $U$ defined in \eqref{eq_uni_char_def}. On removing the factor containing the integral over $U(F)\bk U(\A)$ we obtain

\begin{equation*}
\Pe(\theta^{0}(f,\phi),\chi)=\displaystyle\int_{\Ax T(F)\bk T(\A)} \int_{H^{0}_{1}(F)\bk H^{0}_{1}(\A)}\,\sum_{x\in X^{2}(F)}\,\phi(h_{g}^{-1}h_{0}^{-1}xg)\,f(h_{0}h_{g})\,\Phi(x)\,dh_{0}\,dg
\end{equation*}
where we have introduced the notation

\begin{equation*}
\Phi(x)=\displaystyle\int_{U(F)\bk U(\A)}\psi_{M_{x}}(u)\psi^{-1}_{S}(u)\,du\,.
\end{equation*}
This integral of orthogonal characters simply boils down to

\begin{equation}\label{eq_Phi_orthogonality}
\Phi(x)=\left\lbrace\begin{array}{cl}\vspace{0.05in}
\Vol(U(F)\bk U(\A))&\psi_{M_{x}}=\psi_{S}\\
0&{\rm otherwise}\,.
\end{array}\right.
\end{equation}
The group $U$ is abelian (and hence unimodular) so the Tamagawa number is immediately $\Vol(U(F)\bk U(\A))=1$ (see \cite{weil_tam}). Writing $u=u(A)$ for $A\in\Msymt(\A)$ we then have

\begin{equation*}
\begin{array}{rcl}\vspace{0.05in}
\Phi(x)=1& \Longleftrightarrow& \psi(\Tr(SA-M_{x}A))=1\\
&\Longleftrightarrow& M_{x}=S\,.
\end{array}
\end{equation*}
Thus $\Phi(x)$ is an indicator function allowing only those $x\in X^{2}(F)$ with $M_{x}=S$ to contribute non-zero terms to the summation in $\Pe(\theta^{0}(f,\phi),\chi)$. Define

\begin{equation*}
X^{2}_{S}=\left\lbrace\, x\in X^{2}\,:\, M_{x}=S\,\right\rbrace
\end{equation*}
so that

\begin{equation*}\label{eq_period2}
\Pe(\theta^{0}(f,\phi),\chi)=\displaystyle\int_{\Ax T(F)\bk T(\A)} \int_{H^{0}_{1}(F)\bk H^{0}_{1}(\A)}\,\sum_{x\in X_{S}^{2}(F)}\,\phi(h_{g}^{-1}h^{-1}xg)\,f(hh_{g})\,dh\,dg\,.
\end{equation*}

We are interested in decomposing the algebra $B\isom D\otimes E$ into its subalgebras, in particular the role played by the field $L\isom K\otimes E$. Hence we make the following observation.

\begin{prop}\label{prop_L_emblong_B}
If $L$ does not embed into $B$ as a subalgebra, then $X_{S}^{2}(F)=\emptyset$ and consequently
\begin{equation*}
\Pe(\theta^{0}(f,\phi),\chi)=0\,.
\end{equation*}
\end{prop}

\begin{proof}
Suppose $L\not\emb B$ and assume the contrary: there exists $\xi\in X_{S}^{2}(F)$ with $\xi\neq 0$. Then $\xi$ gives a realisation of $W_{1}$ as a quadratic subspace of $X$ and we have $X=W_{1}\oplus W_{1}^{\perp}$ as before. Since $E\cap X = F$ we have that $X\otimes E \isom B$, so we may decompose $B$ as
\begin{equation*}
B= (W_{1}\otimes E)\oplus (W_{1}^{\perp}\otimes E).
\end{equation*}
But Lemma \ref{lem_W1_isom_K} gives us that $W_{1}=Kw$ for any $w\in W_{1}$. Noting that $1\in X$ we proceed by checking two cases: Firstly, if $1\in W_{1}$ we may take $w=1$ so that $W_{1}=K$. Then $W_{1}\otimes E= L$ and $B=L\oplus L^{\perp}$. Thus $L\emb B$ as a quadratic subalgebra (over $E$), a contradiction. Secondly, if $1\not\in W_{1}$ then $J=W_{1}^{\perp}\otimes E$ is a field and subalgebra of $B$. In fact this field has to be $L$: for any $j\in J^{\perp}=W_{1}\otimes E$ we may write $J^{\perp} = Jj$ but $W_{1}=Kw$ implies $Jj=Lw$ for any $w\in W_{1}\subset J^{\perp}$. Taking $j=w$ gives $J=L$ and thus, once again, we have the contradiction $L\emb B$.
\end{proof}

\begin{assump}\label{assump_on_embedding}
Without loss in generality we assume that $X_{S}^{2}(F)\neq\emptyset$.
\end{assump}

Indeed it is clear from \eqref{eq_Phi_orthogonality} that $X_{S}^{2}(F)=\emptyset$ implies $\Pe(\theta^{0}(f,\phi),\chi) = 0$. Under Assumption \ref{assump_on_embedding} we may conclude that, by Proposition \ref{prop_L_emblong_B}, one has an algebra-embedding $L\emb B$ and subsequently that $K\emb D$ as a subalgebra too. Note that this assumption is truly on the choice of $K$ (or equivalently $d$) since $E$ has been fixed in advance.

We continue by expressing $X^{2}_{S}$ in terms of the group $\SO(X)$ acting on it, reconsidering points of $X^{2}_{S}$ via the isomorphism $X^{2}(F)\isom\Hom_{F}(W_{1},X)$. Fix a base point $\xi\in X_{S}^{2}(F)$, to be considered as an $F$-homomorphism $\xi\colon W_{1}\rightarrow X$ satisfying the properties:

\begin{itemize}\vspace{0.05in}
\item[(1)] $\xi$ is injective (since the Gramm matrix $M_{\xi}=S$ is invertible).\vspace{0.05in}
\item[(2)] $\xi$ is an isometry onto its image in $X$.\vspace{0.05in}
\end{itemize}
We briefly justify (2). Recall that (\S \ref{sec_the_torus}) $W_{1}$ is endowed with the quadratic form $q_{S}$; a simple calculation shows that for $w\in W_{1}$ we have $q_{S}(w)=q_{M_{\xi}}(w)=\N_{X}(\xi(w))$. Thus $W_{1}$ is identified with a quadratic subspace of $X$ via $\xi$. (We abuse notation and call this subspace $W_{1}$ too.) Consider the orthogonal decomposition
\begin{equation*}
X=W_{1}\oplus W_{1}^{\perp}\,.
\end{equation*}

\begin{lem}\label{lem_W1_isom_K}
The image of $W_{1}$ in $X$ is a one-dimensional $K$-vector space: for any $w\in W_{1}$ we have $W_{1}=Kw$. In particular there is an $F$-vector space isomorphism $W_{1}\isom K$.
\end{lem}

\begin{proof}
Recall $M_{\xi}=S$ and fix
\begin{equation*}
S=\begin{pmatrix}
a&b/2\\
b/2&c
\end{pmatrix}\in\Msymt(F)
\end{equation*}
so that $d=-4\det S=b^{2}-4ac$. Fix a basis $\lbrace e_{1},e_{2}\rbrace$ of $W_{1}$ and let $\xi_{i}=\xi(e_{i})$ for $i=1,2$. We show that any two vectors in $W_{1}$ are linearly dependant over $K$. Note that the polynomial $p(X)=X^{2} - bX + ac$ has the root $\xi_{2}\xi_{1}\inv=\frac{1}{2}(b-\sqrt{\delta})$. Multiplying each side by $\xi_{1}$, and noting $a=\N_{X}(\xi_{1})$ by assumption, we see that
\[\xi_{2}=\frac{1}{2a}(b-\sqrt{\delta})\xi_{1}\in K\xi_{1}\, .\]
Since $\xi$ is injective, $\xi_{1}$ and $\xi_{2}$ constitute a basis for $W_{1}\subset X$ over $F$. Hence the $K$-span of any vector $w\in W_{1}$ is equal to $W_{1}$ as $F$-vector spaces.
\end{proof}

We proceed by continuing to exploit the base point $\xi$. The group $\SO(X)$ acts transitively on $X_{S}^{2}(F)$ in which the stabiliser of $\xi$ is $\SO(W_{1}^{\perp})$ by construction. Then after some calculation the isomorphism
\begin{equation*}
X^{2}_{S}(F)\isom\SO(W_{1}^{\perp})\bk \SO(X)
\end{equation*}
permits the following reformulation of \eqref{eq_period2}:

\begin{equation*}
\Pe(\theta^{0}(f,\phi),\chi)=\displaystyle\int_{\SO(W_{1}^{\perp})(\A)\bk\SO(X)(\A)} \phi(h_{0}^{-1}\xi)\, \Lambda_{\xi}(R(h_{0})f,\chi)\,dh_{0}\,,
\end{equation*}
by defining

\begin{equation*}
\Lambda_{\xi}(f,\chi) = \displaystyle\int_{\Ax\G(\SO(W_{1}^{\perp})\times\SO(W_{1}))\bk\G(\SO(W_{1}^{\perp})\times\SO(W_{1}))(\A)} \chi(g)\, f (y_{g}h_{g})\,dy_{g}\, dg
\end{equation*}
where $h_{g}\in H^{0}(\A)$ is any element such that
\begin{equation*}
\lambda(h_{g})=\lambda(g)
\end{equation*}
with the additional constraints that $h_{g}(\xi(v))=\xi(g(v))$ for $v\in W_{1}$ and $h_{g}(w)=w$ when $w\in W_{1}^{\perp}$. The variable of integration $(y_{g},g)$ is an element of $\G(\SO(W_{1}^{\perp})\times\SO(W_{1}))(\A)$ whence $\lambda(y_{g})=\lambda(g)$.

\subsection{Exploiting exceptional isomorphisms}\label{sec_global2}

In this section we analyse the domain of $\Lambda_{\xi}(f,\chi)$ and apply the representation theory of $H$ to rewrite this integral as a period of automorphic forms on $\Bx(\A)$. By the decomposition $X=W_{1}\oplus W_{1}^{\perp}$, we look to reinterpret the subgroup $\G(\SO(W_{1}^{\perp})\times\SO(W_{1}))\leqslant\GSO(X)$ (featured in $\Lambda_{\xi}(f,\chi)$) as a subgroup of $\Fx\times\Bx\,/\Delta\Ex$ via the isomorphism $\rho$ of \eqref{eq_GSO(X)_isom}.

\subsubsection{Structural decomposition of quadratic spaces}

Since $K\emb D$ the standard involution $\inv$ on $D$ restricts to the non-trivial Galois automorphism of $K$. We may write
\begin{equation*}
D=K\oplus Kj
\end{equation*}
for any $j\in K^{\perp}$ since for such a $j$ we have $K^{\perp}=Kj$. Extending this decomposition to $B\isom D\otimes E$ (where $\inv$ extends to a Galois action on $B$, trivial on $E$, as in \S\ref{sec_fourdim}) define
\begin{equation*}
L=K(E)\isom K\otimes_{F} E\,.
\end{equation*}
Then $L=E(\sqrt{d}\,)$ is a quadratic extension of $E$ such that we have an embedding $L\emb B$. The standard involution on $B$ (given by $x\mapsto x\inv$) restricts to the non-trivial Galois involution on $L$. Then, for the same $j\in K^{\perp}$ as before, we have $B=L\oplus L\,j$.

Focusing now on the subspace $X\subset B$ define
\begin{equation*}
X_{L}=\left\lbrace\,x\in L \,:\, \iota(x)=x\inv\,\right\rbrace.
\end{equation*}
Both $X_{L}\subset X$ and $1\in X_{L}$. Moreover, we may realise $X_{L}$ as a quadratic extension of $F$. Under the quadratic form $\N_{X}$ we have the orthogonal decomposition $X=X_{L}\oplus X_{L}^{\perp}$ which is described by the following lemma.

\begin{lem}\label{lem_K_subspace_X}
For any $z_{0}\in E$ with $\Tr_{E/F}(z_{0})=0$ we have the orthogonal decomposition
\begin{equation*}
X=X_{L}\oplus z_{0}\,Kj\,.
\end{equation*}
\end{lem}

\begin{proof}
The orthogonal complement $X_{L}^{\perp}$ is given by $X\cap L\,j$ (otherwise $X_{L}\cap X_{L}^{\perp}\neq 0$). Hence $X_{L}^{\perp}$ contains elements $xj$ where $x\in L$ such that $\iota(xj)=(xj)\inv$; these are the elements $x\in L$ such that $x+\iota(x)=0$ since $j$ and $x$ are orthogonal under $\N_{X}$. Fix some $z_{0}\in E$ with $\Tr_{E/F}(z_{0})=0$ then for any $k\in K$ we have $\iota(z_{0}k)=-z_{0}k$. Hence $z_{0}Kj\subseteq X_{L}^{\perp}$, and since both are two-dimensional $F$-vector spaces we have equality.
\end{proof}

Lemma \ref{lem_W1_isom_K} gave us an interpretation of $W_{1}\subset X$ as the space $W_{1}\isom K$. Combining this with Lemma \ref{lem_K_subspace_X} allows one to deduce the following ($F$-vector space) isomorphisms:

\begin{equation*}
W_{1}\isom X_{L}^{\perp}\quad {\rm and}\quad W_{1}^{\perp}\isom X_{L}\,.
\end{equation*}
Consequently we have the reinterpretation of the orthogonal groups

\begin{equation*}
\GSO(X_{L}^{\perp})\isom\GSO(W_{1})\isom \Kx\quad {\rm and}\quad  \GSO(X_{L})\isom\GSO(W_{1}^{\perp})\,,
\end{equation*}
justifying our conclusion that

\begin{equation}\label{eq_double_gp_isom}
\G(\SO(X_{L})\times\SO(X_{L}^{\perp}))\isom \G(\SO(W_{1}^{\perp})\times\SO(W_{1}))\,.
\end{equation}

\begin{prop}\label{prop_isom_Phi}
There is an $F$-isomorphism of algebraic groups

\begin{equation*}
\Phi\colon \Fx\times \Lx\,/\Delta\Ex\arrup{\longrightarrow}{\sim}\G(\SO(W_{1}^{\perp})\times\SO(W_{1}))
\end{equation*}
where the projection onto the second component is given by $(s,k)\longmapsto s^{-1}kk^{\iota}\in \Kx$ (whereby $\Kx$ acts on $W_{1}\isom K$ by left multiplication).
\end{prop}

\begin{proof}
By the isomorphism \eqref{eq_double_gp_isom}, it suffices to find an isomorphism $\Phi$ such that the following diagram commutes.

\begin{equation}\label{dia_subgp}
\begin{array}{ccc}\vspace{0.08in}
\displaystyle \Fx\times\Bx\,/\Delta\Ex &\displaystyle\arrup{\longrightarrow}{\rho} &\GSO(X)=\GSO(X_{L}\oplus X_{L}^{\perp})\\\vspace{0.08in}
\cup&&\cup\\\vspace{0.08in}
\displaystyle \Fx\times\Lx\,/\Delta\Ex &\displaystyle\arrup{\longrightarrow}{\Phi} &G(\SO(X_{L})\times\SO(X_{L}^{\perp}))
\end{array}
\end{equation}
We consider the surjective map

\begin{equation*}
\begin{array}{ccc}
\Phi\colon \Fx \times \Lx &\displaystyle\longrightarrow& \left\lbrace\, (s^{-1}k\,\iota(k)\inv\,,\,s^{-1}k\,\iota(k)\,)\,:\, s\in\Fx, k\in \Lx\,\right\rbrace\,.
\end{array}
\end{equation*}
One can check that the projections of $\im(\Phi)$, onto the first and second components, act on $X_{L}$ and $X_{L}^{\perp}$, respectively, by left multiplication. Noting that the similitude factors of each component in the image are equal, hence we may extend $\Phi$ to a mapping into $\G(\SO(X_{L})\times\SO(X_{L}^{\perp}))$. Since the kernel of $\Phi$ is $\Delta\Ex$ we have an injection

\begin{equation*}
\begin{array}{ccc}
\Phi\colon \Fx \times \Lx\,/\Delta\Ex &\displaystyle\emblong& \G(\SO(X_{L})\times\SO(X_{L}^{\perp}))\,.
\end{array}
\end{equation*}
To demonstrate the surjectivity of $\Phi$ we need only check that diagram \eqref{dia_subgp} commutes. Observe that, for $(s,k)\in \Fx \times \Lx$,

\begin{equation*}
\begin{array}{rcl}\vspace{0.05in}
\rho(s,k)\bigl(X_{L}\oplus X_{L}^{\perp}\bigr) &=& s^{-1}k\,\bigl(X_{L}\oplus X_{L}^{\perp}\bigr)\,\iota(k)\inv\\\vspace{0.05in}
&=&s^{-1}k\,\iota(k)\inv\, X_{L}\oplus s^{-1}k\,\iota(k)\, X_{L}^{\perp}\\
&=&\Phi(s,k) \bigl(X_{L}\oplus X_{L}^{\perp}\bigr)\,.\end{array}
\end{equation*}
Thus meaning that, up to an automorphism of $X_{L}\oplus X_{L}^{\perp}$, $\rho \vert_{\Fx\times\Lx}= \Phi$. Since $\rho$ is one-to-one then $\Phi$ must also be surjective.
\end{proof}

\subsubsection{Interpretation of the integral $\Lambda_{\xi}(f,\chi)$}\label{sec_interpret_lambda}

Considering the domain of $\Lambda_{\xi}(f,\chi)$, one uses Proposition \ref{prop_isom_Phi} to deduce the isomorphism
\begin{equation*}
\Ax\G(\SO(W_{1}^{\perp})\times\SO(W_{1}))\,\bk\, \G(\SO(W_{1}^{\perp})\times\SO(W_{1}))(\A)\isom\Ax_{E}\,\Lx\bk\Ax_{L}\,.
\end{equation*}
The application of this isomorphism to $\Lambda_{\xi}(f,\chi)$ requires a change of integration variable. This is accomplished by substituting
\begin{equation*}
(y_{g},g)\longmapsto \rho(1,k)
\end{equation*}
where $k\in \Ax_{E}\,\Lx\bk\Ax_{L}$. For this we note that the original variables $h_{g}\in \GSO(W_{1})(\A)$ and $(y_{g},g)\in\G(\SO(W_{1}^{\perp})\times\SO(W_{1}))(\A)$ satisfy:
\begin{itemize}\vspace{0.05in}
\item $h_{g}\in GSO(W_{1})(\A)$ fixes $W_{1}^{\perp}(\A)$ and acts as $g$ on $W_{1}(\A)$,\vspace{0.05in}
\item $y_{g}\in GSO(W_{1}^{\perp})(\A)$ fixes $W_{1}(\A)$ and acts as $y_{g}$ on $W_{1}^{\perp}(\A)$,\vspace{0.05in}
\item $\lambda(h_{g})=\lambda(y_{g})$.\vspace{0.05in}
\end{itemize}
Hence the product $y_{g}h_{g}$, corresponding to $(y_{g},g)$, is substituted with $\rho(1,k)$ and element $g\in\GSO(W_{1})(\A)$, the projection of $(y_{g},g)$ onto its second factor, is substituted with $kk^{\iota}$ (as in Proposition \ref{prop_isom_Phi}). This substitution returns

\begin{equation*}
\Lambda_{\xi}(f,\chi) = \displaystyle\int_{\Ax_{E}\,\Lx\bk\Ax_{L}} \chi(kk^{\iota})\, f(\rho(1,k))\,dk\,.
\end{equation*}
For any $k\in\Ax_{L}$ we have $kk^{\iota}\in\Ax_{K}$ so we have a character $\Omega\colon \Lx\bk\Ax_{L}\rightarrow\Cx$ by defining
\begin{equation}\label{eq_kappa_def}
\Omega(k)=\chi(kk^{\iota})
\end{equation}
such that upon restricting $\Omega$ to $\Ax_{E}$ we have $\Omega\vert_{\Ax_{E}}=\chi\circ\N_{E/F}$. Since we have chosen $f\in\V_{\sigma,S}^{1}$ to correspond to some $\eta\in\V_{\tau}$ such that $f\vert_{H^{0}(\A)}(\rho(s,a))=\eta(a)$, the integral above becomes
\begin{equation}\label{eq_lambda_Q}
\Lambda_{\xi}(f,\chi) =\displaystyle\int_{\Ax_{E}\,\Lx\bk\Ax_{L}} \Omega(k)\, \eta(k)\,dk\,.
\end{equation}

\section{Local Calculation: Integrals Over Matrix Coefficients}\label{sec_local}

We will ultimately show that $\abs{\Pe(\theta(f,\phi),\chi)}^{2}$ factorises into a product of special $L$-values and a finite number of local integrals. In this section we follow \cite{liu} in defining these local integrals and make use of the excellent results proved by Liu to rearrange them for our purposes. Throughout this section we work locally at a place $v$ of $F$ suppressing the subscript $v$ form the notation (so that $F=F_{v}$, $\sigma$ denotes one local component in the tensor product $\otimes_{v}\sigma_{v}$ and so on).

\subsection{Local integrals}\label{sec_stable}

To provide a complete picture, we define the local integrals in full generality for any (local) irreducible, admissible representation $\pi$ of $G$. The definition is divided into a non-archimedean and an archimedean case; this is due to the nature of the analysis in \cite[\S 3]{liu} in `regularising' these integrals. Immediately after this definition we specialise to choosing $\pi=\theta(\sigma)$, the (local) theta lift of $\sigma$, and unify the integrals from each case since they have the same form in this specialisation. We point out that such a $\pi=\theta(\sigma)$ is always tempered and thus the regularisation results of \cite{liu} apply.

\subsubsection{The non-archimedean case} Suppose that $F$ is a non-archimedean local field. We consider the notion of a stable integral as defined in \cite{lapid_mao}. We refer the reader to there for more information since it is not of central importance to our discussion.

\begin{defn}[The non-archimedean local factors]\label{def_finite_local_factors}
Given $\varphi\in\V_{\pi}$, $\tilde{\varphi}\in\V_{\bar{\pi}}$ and a unitary paring
\begin{equation*}
\B_{\pi}\colon \V_{\pi}\otimes\V_{\bar{\pi}}\rightarrow\C
\end{equation*}
we define
\begin{equation*}
\alpha(\varphi,\tilde{\varphi};\chi)=\int_{\Fx\bk T} \,\int_{U}^{{\rm st}}\B_{\pi}( \pi(ug)\varphi,\tilde{\varphi})\,\chi(g)\,\psi^{-1}_{S}(u)\, du\,dg
\end{equation*}
where the integral over $U$ is called a \textit{stable integral} (see \cite[Definition 2.1]{lapid_mao}) and is evaluated on a certain compact open subgroup $N\subset U$. This $N$ is chosen to be `maximally' in the sense that if $N'$ is another compact open subgroup with $N\subset N' \subset U$ then the integral over $N'$ equals the integral over $N$. The product of Haar measures $du\,dg$ is again a Haar measure on the Bessel subgroup $\Fx\bk R$.
\end{defn}

Indeed it is not obvious that the integrals of Definition \ref{def_finite_local_factors} converge, nor should such an $N$ exist, but Liu proves these facts in \cite[Theorem 2.1]{liu} and \cite[Lemma 3.2]{liu}, respectively.

\subsubsection{The archimedean case} Let $F$ be an archimedean local field. The method of regularisation here is to consider the Fourier transform of certain matrix coefficients in a so-called \textit{regular} subset of $U$.

Recall that the abelian unipotent group $U\isom\Msymt(F)$ is self-dual and all its characters are given by $\psi_{M}$, for some $M\in \Msymt(F)$, as in \ref{eq_uni_char_def}. We denote by $\Msymt(F)^{\rm reg}$ the open and dense subset of non-singular symmetric matrices in $\Msymt(F)$ and define its image in $U$ as
\begin{equation*}
U^{\rm reg}\isom \Msymt(F)^{\rm reg}\,.
\end{equation*}

\begin{defn}[The archimedean local factors]\label{def_infinite_local_factors}
Given $\varphi\in\V_{\pi}$, $\tilde{\varphi}\in\V_{\bar{\pi}}$ and a unitary paring
\begin{equation*}
\B_{\pi}\colon \V_{\pi}\otimes\V_{\bar{\pi}}\rightarrow\C
\end{equation*}
we define
\begin{equation*}
\alpha(\varphi,\tilde{\varphi};\chi)=\int_{\Fx\bk T} \,\int_{U^{\rm reg}}\B_{\pi}( \pi(ug)\varphi,\tilde{\varphi})\,\chi(g)\,\psi_{S}^{-1}(u)\, du\,dg\,.
\end{equation*}
Here, for a fixed $g\in T$, the map

\begin{equation*}
\psi_{S}\longmapsto\int_{U^{\rm reg}}\B_{\pi}( \pi(ug)\phi,\tilde{\phi})\,\psi_{S}^{-1}(u)\, du
\end{equation*}
is the Fourier transform (in $U^{\rm reg}$) of the function $u\longmapsto\B_{\pi}(\pi(ug)\theta(f,\phi),\theta(\tilde{f},\tilde{\phi}))$.
\end{defn}

Once again, Liu proves that this integral converges absolutely in \cite[Theorem 2.1]{liu}.

\subsubsection{Normalisation of local integrals}

In his paper \cite{liu}, Liu goes on to show that there exists a specified set of \textit{good} places, which exclude a \textit{finite} number of places of the base number field (including the archimedean ones), for which the local integrals may be computed as follows (see \cite[p.~7]{liu} for details).

\begin{prop}\label{prop_normalisation} If $v$ is a \textit{good} place of the base number field then for the local vectors $\varphi\in\V_{\pi}$, $\tilde{\varphi}\in\V_{\bar{\pi}}$ one has
\begin{equation*}
\alpha(\varphi,\tilde{\varphi};\chi)=\dfrac{\zeta_{F}(2)\,\zeta_{F}(4)\,L(1/2,\pi\boxtimes\chi)}{L(1,\pi,{\rm Ad})\, L(1,\chi_{K/F})}\,.
\end{equation*}
\end{prop}
Hence we normalise the local factors by setting

\begin{equation}\label{eq_nomalise_complicated}
\alpha^{\natural}(\varphi,\tilde{\varphi};\chi)= \dfrac{L(1,\pi,{\rm Ad})\, L(1,\chi_{K/F})}{\zeta_{F}(2)\,\zeta_{F}(4)\, L(1/2,\pi\boxtimes\chi)}\,\alpha(\varphi,\tilde{\varphi};\chi)
\end{equation}
so that $\alpha^{\natural}(\varphi,\tilde{\varphi};\chi)=1$ for almost all $v$.

Given \textit{any} place $v$, if, instead of considering an arbitrary vector $\tilde{\varphi}\in\V_{\bar{\pi}}$, we take the local vector $\tilde{\varphi}=\bar{\varphi}$ -- in the context of being local factors of functions on ad\`ele groups as in \eqref{eq_fix_theta_conj_factors} -- then we define the notation
\begin{equation}\label{eq_alpha_short}
\alpha(\varphi\,,\chi)=\alpha(\varphi,\bar{\varphi};\chi)\quad{\rm and} \quad \alpha^{\natural}(\varphi,\chi)=\alpha^{\natural}(\varphi,\bar{\varphi};\chi)\,.
\end{equation}
As well as absolute convergence, \cite[Theorem 2.1]{liu} states that whenever such a $\pi$ is tempered, we have the positivity result
\begin{equation*}
\alpha(\varphi\,,\chi)\geq 0\,.
\end{equation*}

\begin{rem}
The integrals defining $\alpha(\varphi,\tilde{\varphi};\chi)$ have a unipotent part (over $U$) which is given by either a stable integral (over a compact open $N\subset U$) or a Fourier transform (with respect to $U^{\rm reg}\subset U$) when $v$ is non-archimedean or archimedean, respectively. We consider these integrals for $\pi$ tempered. The choices of regularisation for these integrals are justified by noting that when $\pi$ is \textit{square integrable} we may take the entire space $U$ in each definition. That is, for \textit{any} $v$, when $\pi$ is square integrable we have
\begin{equation*}
\alpha(\varphi,\tilde{\varphi};\chi)=\int_{\Fx\bk T}\int_{U}\B_{\pi}( \pi(ug)\varphi,\tilde{\varphi})\,\chi(g)\,\psi^{-1}_{S}(u)\, du\,dg\,,
\end{equation*}
by Propositions 3.5 and 3.15 of \cite{liu}.
\end{rem}

\subsubsection{A unified result for theta lifts}

Let us specialise now by assuming $\pi=\theta(\sigma)$ is the theta lift of $\sigma$, a local factor of the fixed representation in \S \ref{sec_hypotheses}. We select the pairing $\B_{\pi}$ to be defined as in \eqref{eq_local_theta_pair}; this depends on a choice of $\B_{\sigma}$ which we made in \eqref{eq_choose_sigma_pairing}. Retaining some generality in what follows, we note that by \cite[Proposition 5.5]{gan_ichino} the conjugate representation $\bar{\pi}$ is generated by elements $\theta(\tilde{f},\tilde{\phi})$ for $\tilde{f}\in \V_{\bar{\sigma}}$ and $\tilde{\phi}\in\V_{\bar{\weil}}$.

\begin{prop} In either the non-archimedean or archimedean cases, if $\theta(f,\phi)\in\V_{\pi}$ and $\theta(\tilde{f},\tilde{\phi})\in\V_{\bar{\pi}}$ then the local integrals become
\begin{equation*}
\begin{array}{c}\vspace{0.15in}
\alpha(\theta(f,\phi),\theta(\tilde{f},\tilde{\phi});\chi)=\displaystyle\dfrac{\zeta_{F}(2)\,\zeta_{F}(4)}{L(1,\sigma,\std)}\,\int_{\Fx\bk\GSO(W_{1})}\, \int_{\Or(X)}\,\int_{\SO(W^{\perp}_{1})\bk\SO(X)}\hspace{0.9in}\hfill\\
\hfill \times \,\phi (h_{g}^{-1}h^{-1}h_{1}^{-1}\xi\, g)\,\tilde{\phi}(h_{1}^{-1}\xi)\,\B_{\sigma}( \sigma(h h_{g})f,\tilde{f})\,dh_{1}\,dh\,dg
\end{array}
\end{equation*}
where $h_{g}\in H^{0}(\A)$ is any element such that $\lambda(h_{g})=\lambda(g)$ with the additional constraints that $h_{g}(\xi(v))=\xi(g(v))$ for $v\in W_{1}$ and $h_{g}(w)=w$ when $w\in W_{1}^{\perp}$ (for comparison see \S \ref{sec_unfolding}); the element $\xi\in X^{2}_{S}$ is the base point chosen in \S \ref{sec_unfolding}; $dh$ is the Haar measure for $\Or(X)$ fixed in the definition for $\B_{\theta(\sigma)}$, see \eqref{eq_local_theta_pair}; and finally $dh_{1}$ is the Siegel--Weil measure on $\SO(W^{\perp}_{1})\bk\SO(X)$.
\end{prop}

\begin{proof}
This follows immediately from \cite[Lemma 4.2]{liu}.
\end{proof}

\begin{rem}
The product of local Siegel--Weil measures is precisely the Tamagawa measure on the ad\`elic points of the group in question (see \cite[Remark 3.18]{liu}).
\end{rem}

\subsection{Explicit local factors for theta lifts}\label{sec_local_calc} We analyse the terms $\alpha^{\natural}(\theta(f,\phi),\theta(\tilde{f},\tilde{\phi});\chi)$ where $\theta(f,\phi)\in\V_{\pi}$ and $\theta(\tilde{f},\tilde{\phi})\in\V_{\bar{\pi}}$ are as before. We point out again that, even though the subscripts are removed, everything is local here. We will determine the quantity

\begin{equation*}
\begin{array}{c}\vspace{0.07in}
\left(\displaystyle\dfrac{\zeta_{F}(2)\,\zeta_{F}(4)}{L(1,\sigma,\std)}\right)^{-1}\,\alpha(\theta(f,\phi),\theta(\tilde{f},\tilde{\phi});\chi)=\hfill\\\vspace{0.05in}
\hspace{0.9in}\displaystyle\int_{\Fx\bk\GSO(W_{1})}\, \int_{\Or(X)}\,\int_{\SO(W^{\perp}_{1})\bk\SO(X)} \,(\weil(h)\phi) (\xi)\,\,(\bar{\weil}(h_{1})\tilde{\phi})(\xi)\hspace{0.8in}\hfill\\
\hfill\times\,\B_{\sigma}( \sigma(h_{1}^{-1}h_{g}h)f,\tilde{f})\,\chi(g)\,dh_{1}\,dh\,dg
\end{array}
\end{equation*}
after making the substitution $h\mapsto h_{1}^{-1}h_{g}hh_{g}^{-1}$ and recalling that $\xi\,g=h_{g}\,\xi$, by definition. We decompose the integral over $\Or(X)$ in terms of its connected component $\SO(X)$ and replace the measure $dh$ with
\begin{equation*}
dh_{2}=2dh\vert_{\SO(X)}
\end{equation*}
so that the volumes
\begin{equation*}
\Vol(\Or(X),dh)=\Vol(\SO(X),dh_{2})\,.
\end{equation*}
Then we find that the right-hand side of the above quantity is equal to

\begin{equation*}
\begin{array}{c}\vspace{0.05in}
\displaystyle \dfrac{1}{2}\sum_{\varepsilon\in\mu_{2}(F)}\, \int_{\Fx\bk\GSO(W_{1})}\,\int_{\SO(X)}\,\int_{\SO(W^{\perp}_{1})\bk\SO(X)}\, (\weil(h_{2}\varepsilon)\phi) (\xi)\,\,(\bar{\weil}(h_{1})\tilde{\phi})(\xi)\hspace{0.8in}\hspace{0.1in}\\
\hfill\times\,\B_{\sigma}( \sigma(h_{1}^{-1}h_{g}h_{2}\varepsilon)f,\tilde{f})\,\chi(g)\,dh_{1}\,dh_{2}\,dg\,.
\end{array}
\end{equation*}
To simplify further, note that
\begin{equation*}
\SO(X)\isom(\SO(W^{\perp}_{1})\bk\SO(X))\,\times\,\SO(W_{1}^{\perp})
\end{equation*}
where we substitute $h_{2}\mapsto (h_{2},y)$, with measure $dh_{2}\mapsto dh_{2}dy$, so that

\begin{equation*}
\begin{array}{c}\vspace{0.07in}
\left(\displaystyle\dfrac{\zeta_{F}(2)\,\zeta_{F}(4)}{L(1,\sigma,\std)}\right)^{-1}\,\alpha(\theta(f,\phi),\theta(\tilde{f},\tilde{\phi});\chi)=\hfill\\\vspace{0.05in}
\hspace{0.3in}\dfrac{1}{2}\displaystyle\sum_{\varepsilon\in\mu_{2}(F)}\,\displaystyle\int_{\Fx\bk\GSO(W_{1})}\, \int_{\SO(W_{1}^{\perp})}\,\int_{(\SO(W^{\perp}_{1})\bk\SO(X))^{2}} \,(\weil(h_{2}\varepsilon)\phi) (\xi)\,\,(\bar{\weil}(h_{1})\tilde{\phi})(\xi)\hspace{0.2in}\\
\hfill\times\,\B_{\sigma}( \sigma(h_{1}^{-1}yh_{g}h_{2}\varepsilon)f,\tilde{f})\,\chi(g)\,dh_{1}\,dh_{2}\,dy\,dg\,,
\end{array}
\end{equation*}
recalling that $y\in\SO(W_{1}^{\perp})$ stabilises $\xi$ and commutes with $h_{g}$. Using that $\sigma$ is unitary under $\B_{\sigma}$ we finally obtain

\begin{equation*}
\begin{array}{c}\vspace{0.14in}
\alpha(\theta(f,\phi),\theta(\tilde{f},\tilde{\phi});\chi)=\dfrac{1}{2}\,\dfrac{\zeta_{F}(2)\,\zeta_{F}(4)}{L(1,\sigma,\std)}\displaystyle\,\sum_{\varepsilon\in\mu_{2}(F)}\,\displaystyle\int_{(\SO(W^{\perp}_{1})\bk\SO(X))^{2}} \,(\weil(h_{2}\varepsilon)\phi) (\xi)\hfill\\
\hspace{2.5in}\times\,(\bar{\weil}(h_{1})\tilde{\phi})(\xi)\,\Gamma_{\xi,v}(\sigma(h_{2}\varepsilon)f,\bar{\sigma}(h_{1})\tilde{f};\chi)\,dh_{1}\,dh_{2}
\end{array}
\end{equation*}
by defining
\begin{equation*}
\Gamma_{\xi,v}(f,\tilde{f};\chi)=\displaystyle\int_{\Fx\bk\GSO(W_{1})}\,\int_{\SO(W_{1}^{\perp})} \,\B_{\sigma}( \sigma(yh_{g})f,\tilde{f})\,\chi(g)\,dy\,dg\,.
\end{equation*}

\section{The Result: Local and Global Assembly}\label{sec_final}
This section concludes with the unification of the global period in \S \ref{sec_global} and the rearranged local integrals in \S \ref{sec_local}. The connection is facilitated by the work of Waldspurger \cite{waldspurger} who, in 1985, gave the pioneering example of refined Gan--Gross--Prasad conjecture: a proof for the pair $(\SO_{3},\SO_{2})$. We apply his formula to our calculation.

\subsection{A theorem of Waldspurger}\label{sec_thm_of_walds} Let $B$ be a (possibly split) quaternion algebra over $E$. Let $L$ be a quadratic extension of a number field $E$ such that there exists an embedding $L\emb B$ and let $\Omega$ be a Hecke character of $\Ax_{L}$. Let $\tau=\otimes_{w}\tau_{w}$ be an irreducible, cuspidal automorphic representation of $\Bx(\A_{E})$, realised in $\V_{\tau}$, such that $\omega_{\tau}\cdot\Omega\vert_{\Ax_{E}}=1$. For $\eta\in\V_{\tau}$ define the global period integral

\begin{equation*}
\Qe(\eta,\Omega)=\int_{\Ax_{E}\, \Lx\,\bk\,\Ax_{L}}\Omega(k)\,\eta(k)\,dk\,.
\end{equation*}
For each place $w$ of $E$ let $\B_{\tau_{w}}$ be a unitary pairing on $\V_{\tau_{w}}\otimes\V_{\bar{\tau}_{w}}$. For each $\eta_{w}\in\V_{\tau_{w}}$ and $\tilde{\eta}_{w}\in\V_{\bar{\tau}_{w}}$ define the local integrals

\begin{equation*}
\beta_{w}(\eta_{w},\tilde{\eta}_{w};\Omega_{w})=\displaystyle\int_{\Ex_{w}\bk \Lx_{w}}\B_{\tau_{w}}(\tau_{w}(k_{w})\eta_{w},\tilde{\eta}_{w})\,\Omega_{w}(k_{w})\,dk_{w}
\end{equation*}
and their natural normalisation,

\begin{equation*}
\beta_{w}^{\natural}(\eta_{w},\tilde{\eta}_{w};\Omega_{w})=\dfrac{L(1,\tau_{w},\Ad)\,L(1,\chi_{L_{w}/E_{w}})}{\zeta_{E_{w}}(2)\,L(1/2,\tau_{L,w}\otimes\Omega_{w})}\,\beta_{w}(\eta_{w},\tilde{\eta}_{w};\Omega_{w})\,.
\end{equation*}
where $\tau_{L,w}$ is the base change lift of $\tau_{w}$ to $\Bx(L_{w})$.

The following theorem was originally given in \cite[\S III.3]{waldspurger} (and then stated in terms of the refined Gan--Gross--Prasad conjecture in \cite[\S 6]{ichino_ikeda}). Fix a choice of Haar measures $dk_{w}$, such that the Tamagawa measure on $(\Ex\bk\Lx) (\A)$ decomposes as $dk=\prod_{w}dk_{w}$, and a choice of local parings $\B_{\tau_{w}}$, such that the Petersson inner product decomposes as $\B_{\tau}=\prod_{w}\B_{\tau_{w}}$.

\begin{theorem}[Waldspurger]\label{thm_walds}
The integrals $\beta_{w}(\eta_{w},\tilde{\eta}_{w};\Omega_{w})$ are absolutely convergent and
\begin{equation*}
\beta_{w}^{\natural}(\eta_{w},\tilde{\eta}_{w};\Omega_{w})=1
\end{equation*}
for almost all places $w$ of $E$. If, in addition, $\tau$ has trivial central character ($\omega_{\tau}=1$) and $\Omega$ is unitary then
\begin{equation*}
\Qe(\eta,\Omega)\Qe(\tilde{\eta},\bar{\Omega})=\dfrac{1}{2}\dfrac{\zeta_{E}(2)\,L(1/2,\tau_{L}\pr\otimes\Omega)}{L(1,\tau,\Ad)\,L(1,\chi_{L/E})}\,\prod_{w}\,\beta_{w}^{\natural}(\eta_{w},\tilde{\eta}_{w};\Omega_{w})
\end{equation*}
where $\tau_{L}$ denotes the base change lift of $\tau$ to $\Bx(\A_{L})$ and $\tau_{L}\pr$ is the Jacquet--Langlands transfer of $\tau_{L}$ to $\GLt(\A_{L})$.
\end{theorem}

We remark that the $L$-function $L(1/2,\tau_{L}\pr\otimes\Omega)$ may be interpreted in various ways due to the low-dimensional isomorphisms that occur (see \S \ref{sec_l_functions}).

\subsection{Application of Waldspurger}

Let the arbitrary notation introduced in \S \ref{sec_thm_of_walds} now assume the running meanings that we assigned in \S \ref{sec_hypotheses} (for the representation $\tau=\otimes_{w}\tau_{w}$ and the pairings $\B_{\tau_{w}}$) and \S \ref{sec_global2} (for the algebras $B\isom D\otimes E$, $L\isom K\otimes E$). We draw special attention to the assumption that $f\in\V_{\sigma,S}^{1}$ with $f\vert_{H^{0}(\A)}\circ\rho=\eta$. The set $\Set$ contains those places of $F$ such that $\sigma_{v}\isom\sigma_{v}\otimes\sgn$ and $S$ is the fixed, finite set of places of $F$ outside which $f_{v}=f^{\circ}_{v}$ is $H(\mathcal{O}_{v})$-invariant (see \S \ref{sec_autm_forms_H}). We choose $\eta=\otimes_{w}\eta_{w}$, implying $f=\otimes_{v}f_{v}$ with $f_{v}=\otimes_{w\vert v}\eta_{w}$ as in \S \ref{sec_factorise_reps}. The pairings $\B_{\tau_{w}}$, for $w\vert v$, determine the pairings $\B_{\sigma_{0,v}}$ and $\B_{\sigma_{v}}$ (as in \S \ref{sec_pairings}) which are used to define the local integrals (\S \ref{sec_local}).

\begin{lem}\label{lem_global_factors_id} The global period integral in Waldspurger's formula satisfies

\begin{equation*}
\Lambda_{\xi}(f,\chi) = \Qe(\eta,\Omega)\,.
\end{equation*}
\end{lem}

\begin{proof}
We only need to remark that $\Omega\vert_{\Ax_{E}}=\chi\circ\N_{E/F}$, implying the condition $\Omega\vert_{\Ax_{E}}=1$ is satisfied since $\chi\vert_{\Ax}=1$ (Assumption \ref{assump_on_K_chi}). Moreover, $\Omega$ is unitary because $\chi$ is assumed so. We then have that the form of $\Lambda_{\xi}(f,\chi)$ in \eqref{eq_lambda_Q} is given precisely by $\Qe(\eta,\Omega)$.
\end{proof}

In a similar manner, we identify the local period integrals in Waldspurger's formula with our own terms $\Gamma_{\xi,v}(f_{v},\tilde{f}_{v};\chi_{v})$. The following lemma is a local analogue of the analysis of $\Lambda_{\xi}(f,\chi)$ in \S \ref{sec_global2}.

\begin{lem}\label{lem_local_factors_id}
Let $v$ be a place of $F$. Then, for $f_{v}\in\V_{\sigma_{v}}$ and $\tilde{f}_{v}\in\V_{\bar{\sigma}_{v}}$ as above,

\begin{equation*}
\Gamma_{\xi,v}(f_{v},\tilde{f}_{v};\chi_{v})=\dfrac{1}{2^{c_{v}}}\,\prod_{w\vert v}\beta_{w}(\eta_{w},\tilde{\eta}_{w};\Omega_{w})
\end{equation*}
where
\begin{equation*}
c_{v}=\left\lbrace\begin{array}{ll}\vspace{0.05in}
1&{\rm if}\,\, v\in\Set\cap S\\
0&{\rm otherwise}\qquad.
\end{array}\right.
\end{equation*}
\end{lem}

\begin{proof}
Analogous to the global setting (discussed in \S \ref{sec_interpret_lambda}) we have

\begin{equation*}
\Fx_{v}\bk\GSO(W_{1})_{v}\times\SO(W_{1}^{\perp})_{v}\isom \Fx_{v}\bk \G(\SO(W^{\perp}_{1})\times\SO(W_{1}))_{v}
\end{equation*}
so that

\begin{equation}\label{eq_apply_local_isom}
\Gamma_{\xi,v}(f_{v},\tilde{f}_{v};\chi_{v})=\displaystyle\int_{\Fx_{v}\bk \G(\SO(W^{\perp}_{1})\times\SO(W_{1}))_{v}}\,\B_{\sigma_{v}}( \sigma_{v}(yh_{g})f_{v},\tilde{f}_{v})\,\chi_{v}(g)\,dy_{g}\,dg
\end{equation}
where $h_{g}\in H^{0}_{v}$ is any element such that
\begin{equation*}
\lambda(h_{g})=\lambda(g)
\end{equation*}
with the additional constraints that $h_{g}(\xi(v))=\xi(g(v))$ for $v\in W_{1,v}$ and $h_{g}(w)=w$ when $w\in W_{1,v}^{\perp}$. The variable of integration $(y_{g},g)$ is an element of $\G(\SO(W_{1}^{\perp})\times\SO(W_{1}))_{v}$ whence $\lambda(y_{g})=\lambda(g)$. By Proposition \ref{prop_isom_Phi} there is an $F_{v}$-isomorphism

\begin{equation*}
\Fx_{v}\bk \G(\SO(W^{\perp}_{1})\times\SO(W_{1}))_{v}\isom (\Ex\bk\Lx)(F_{v})\,.
\end{equation*}
Applying this isomorphism to \eqref{eq_apply_local_isom} (checking \S \ref{sec_interpret_lambda} for comparison), we substitute the element $y_{g}h_{g}$, which corresponds to $(y_{g},g)$ by definition, with $\rho(1,k)$ where $k\in (\Ex\bk\Lx)(F_{v})$. The element $g\in\GSO(W_{1})_{v}$ is the projection of $(y_{g},g)$ onto its second factor; as in Proposition \ref{prop_isom_Phi}, this projection corresponds to $\rho(1,k)\mapsto kk^{\iota}$. This substitution returns

\begin{equation*}
\Gamma_{\xi,v}(f_{v},\tilde{f}_{v};\chi_{v})=\displaystyle\int_{(\Ex\bk\Lx)(F_{v})}\,\B_{\sigma_{v}}(\sigma_{v}(\rho(1,k_{v}))f_{v},\tilde{f}_{v})\,\chi_{v}(k_{v}k_{v}^{\iota})\,dk_{v}\,.
\end{equation*}

The automorphic character $\Omega=\otimes_{w}\Omega_{w}$ of \eqref{eq_kappa_def}, factorised over places of $E$, may be divided into factors corresponding to each place $v$ of $F$ by $\Omega_{v}=\otimes_{w\vert v}\Omega_{w}$. These factors coincide with the factorisation of $\chi=\otimes_{v}\chi_{v}$ in that $\Omega_{v}\colon k_{v} \mapsto\chi_{v}(k_{v}k_{v}^{\iota})$.

The measures $dk_{v}$ are chosen so that the Tamagawa measure $dk$ on $(\Ex\bk\Lx)(\A_{E})$ factorises as
\begin{equation*}
dk=\prod_{w}dk_{w}\,,
\end{equation*}
over places of $E$, with $dk_{v}=\prod_{w\vert v}dk_{w}$. The $dk_{v}$ are precisely the measures $dh_{1,v}$ of $H_{1,v}$ in \eqref{eq_local_theta_pair} (defining $\B_{\theta(\sigma_{v})}$). We now express the domain in terms of places $w$ of $E$. By \eqref{eq_rama} we have

\begin{equation*}
(\Ex\bk\Lx)(F_{v})\isom  \prod_{w\vert v} \Ex_{w}\bk\Lx_{w}\,.
\end{equation*}

Our calculation now depends on whether or not $v\in\Set$. With the vectors $f_{v}=\otimes_{w\vert v}\eta_{w}$ and $\tilde{f}_{v}=\otimes_{w\vert v}\tilde{\eta}_{w}$ we have

\begin{equation*}
\B_{\sigma_{v}}(f_{v},\tilde{f}_{v})=\frac{1}{2^{c_{v}}}\,\B_{\sigma_{0,v}}(f_{v},\tilde{f}_{v})=\frac{1}{2^{c_{v}}}\,\prod_{w\vert v}\B_{\tau_{w}}(\eta_{w},\tilde{\eta}_{w})\,.
\end{equation*}
This is clear from the definition of the pairing $\B_{\sigma_{v}}$ in \S \ref{sec_pairings} if $v\not\in S$ or $v\not\in\Set$. If $v\in S\cap \Set$ then
\begin{equation*}
f_{v}=f_{v}+0\in\V_{\sigma_{0,v}}\oplus\V_{\sigma_{0,v}^{\iota}}
\end{equation*}
so we pick up the factor of $1/2^{c_{v}}=1/2$.

At last we obtain
\begin{equation*}
\begin{array}{rcl}\vspace{0.1in}
\Gamma_{\xi,v}(f_{v},\tilde{f}_{v};\chi_{v})&=&\displaystyle\int_{ \prod_{w\vert v} \Ex_{w}\bk\Lx_{w}}\,\frac{1}{2^{c_{v}}}\,\prod_{w\vert v}\B_{\tau_{w}}(\tau_{w}(k_{w})\eta_{w},\tilde{\eta}_{w})\,\Omega_{w}(k_{w})\,dk_{w}\\
&=&\displaystyle\frac{1}{2^{c_{v}}}\,\prod_{w\vert v}\,\beta_{w}(\eta_{w},\tilde{\eta}_{w};\Omega_{w})\,.
\end{array}
\end{equation*}
\end{proof}

Combining the previous two lemmas allows Waldspurger's formula to be rewritten in terms of the integrals defining $\Lambda_{\xi}$ and $\Gamma_{\xi}$. Recall the notation $S\pr=S\smallsetminus(S \cap\Set)$ and introduce
\begin{equation*}
s=\abs{S\cap\Set}\quad {\rm and}\quad s\pr=\abs{S\pr}\,.
\end{equation*}

\begin{prop}\label{prop_walds_reform}
For all pure tensors $f=\otimes_{v}f_{v}\in\V_{\sigma,S}^{1}$ and $\tilde{f}=\otimes_{v}\tilde{f}_{v}\in\V_{\bar{\sigma},S}^{1}$ we have
\begin{equation*}
\Lambda_{\xi}(f,\chi)\Lambda_{\xi}(\tilde{f},\bar{\chi})=2^{s-1}\,\prod_{v}\, \Gamma_{\xi,v}(f_{v},\tilde{f}_{v};\chi_{v})\,.
\end{equation*}
\end{prop}

\subsection{The explicit formula}\label{sec_reform_form}

Applying the definition of the variant theta integral \eqref{eq_variant_formula} we begin computing the Bessel period's square:

\begin{equation*}
\begin{array}{l}\vspace{0.15in}
\abs{\Pe(\theta(f,\phi),\chi)}^{2}=\hfill\\
\hspace{0.5in}\displaystyle\int_{\mu_{2}(F)\bk\mu_{2}(\A)}\int_{\mu_{2}(F)\bk\mu_{2}(\A)}\Pe(\theta^{0}(\sigma(\delta) f,\weil(\delta)\phi),\chi)\,\overline{\Pe(\theta^{0}(\sigma(\varepsilon) f,\weil(\varepsilon)\phi),\chi)}\,d\delta\,d\varepsilon\,.
\end{array}
\end{equation*}
As $\mu_{2}(F)$ is of index-two in $\mu_{2}(\A)$ we rearrange so that the above integral equals

\begin{equation}\label{eq_above_summation}
\begin{array}{l}\vspace{0.1in}
\displaystyle\dfrac{1}{4}\int_{\mu_{2}(\A)}\int_{\mu_{2}(\A)}\Pe(\theta^{0}(\sigma(\delta) f,\weil(\delta)\phi),\chi)\,\overline{\Pe(\theta^{0}(\sigma(\varepsilon) f,\weil(\varepsilon)\phi),\chi)}\,d\delta\,d\varepsilon\\
\hspace{0.5in}=\displaystyle\dfrac{1}{4^{1+s+s\pr}}\sum_{\mu_{2}(F_{S})}\sum_{\mu_{2}(F_{S})}\Pe(\theta^{0}(\sigma(\delta) f,\weil(\delta)\phi),\chi)\,\overline{\Pe(\theta^{0}(\sigma(\varepsilon) f,\weil(\varepsilon)\phi),\chi)}\,.
\end{array}
\end{equation}
This equality follows since, as $\varepsilon_{v}\in H(\mathcal{O}_{v})$, the integrals for $v\not\in S$ fix the integrand and elsewhere we have the (normalised) counting Haar measure. We further reduce the sum by noting that, for $h_{0}\in H^{0}(\A)$,
\begin{equation*}
\sigma(\varepsilon)f(h_{0})=f(h_{0}\varepsilon)=0
\end{equation*}
unless $\varepsilon\in\mu_{2}(\A^{S\cap\Set})\lbrace 1,\iota\rbrace$ (by \eqref{eq_gi_lem2.2} or \cite[Lemma 2.2]{gan_ichino}). Hence \eqref{eq_above_summation} equals

\begin{equation*}
\begin{array}{l}
\displaystyle\dfrac{1}{4^{1+s+s\pr}}\sum_{\mu_{2}(F_{S\pr})}\sum_{\mu_{2}(F_{S\pr})}\bigg(\Pe(\theta^{0}(\sigma(\delta) f,\weil(\delta)\phi),\chi)+\Pe(\theta^{0}(\sigma(\delta\iota) f,\weil(\delta\iota)\phi),\chi)\bigg)\hspace{0.5in}\\
\hfill\times\,\bigg(\overline{\Pe(\theta^{0}(\sigma(\varepsilon) f,\weil(\varepsilon)\phi),\chi)}+\overline{\Pe(\theta^{0}(\sigma(\varepsilon\iota) f,\weil(\varepsilon\iota)\phi),\chi)}\bigg)\,.
\end{array}
\end{equation*}
The invariance under $\iota$, noted in \eqref{eq_theta0_inv_iota}, implies we have the equality

\begin{equation*}
\abs{\Pe(\theta(f,\phi),\chi)}^{2}=\displaystyle\dfrac{1}{4^{s+s\pr}}\sum_{\mu_{2}(F_{S\pr})}\sum_{\mu_{2}(F_{S\pr})}\Pe(\theta^{0}(\sigma(\delta) f,\weil(\delta)\phi),\chi)\,\overline{\Pe(\theta^{0}(\sigma(\varepsilon) f,\weil(\varepsilon)\phi),\chi)}\,.
\end{equation*}
Hence it suffices to proceed by considering the summands

\begin{equation*}
\begin{array}{l}\vspace{0.14in}
\Pe(\theta^{0}(\sigma(\delta) f,\weil(\delta)\phi),\chi)\,\overline{\Pe(\theta^{0}(\sigma(\varepsilon) f,\weil(\varepsilon)\phi),\chi)}
\hfill\\\vspace{0.14in}
\hspace{1.2in}=\displaystyle\int\int_{((\SO(W_{1}^{\perp})\bk \SO(X))(\A))^{2}} \,(\weil(h_{2}\delta)\phi)(\xi)\,\,\overline{(\weil(h_{1}\varepsilon)\phi)(\xi)}\hspace{0.9in}\\
\hfill\displaystyle\times\,\Lambda_{\xi}(\sigma(h_{2}\delta)f),\chi)\,\overline{\Lambda_{\xi}(\sigma(h_{1}\varepsilon)f,\chi)}\,dh_{1}\,dh_{2}\,.
\end{array}
\end{equation*}
We have $\overline{\Lambda_{\xi}(\sigma(h_{1}\varepsilon)f,\chi)}=\Lambda_{\xi}(\bar{\sigma}(h_{1}\varepsilon)\bar{f},\bar{\chi})$ where $\bar{f}=\otimes_{v}\bar{f}_{v}\in\V_{\bar{\sigma},S}^{1}$ and the vectors
\begin{equation*}
\sigma(h_{1}\varepsilon)f=\otimes_{v}\sigma_{v}(h_{1,v}\varepsilon_{v})f_{v}\in\V_{\sigma,S}^{1}\quad {\rm and}\quad \bar{\sigma}(h_{1}\varepsilon)\bar{f}=\otimes_{v}\bar{\sigma}_{v}(h_{1,v}\varepsilon_{v})\bar{f}_{v}\in\V_{\bar{\sigma},S}^{1}
\end{equation*}
are pure tensors. Thus the hypotheses of Proposition \ref{prop_walds_reform} are satisfied; we have

\begin{equation*}
\Lambda_{\xi}(\sigma(h_{2}\delta)f),\chi)\,\overline{\Lambda_{\xi}(\sigma(h_{1}\varepsilon)f,\chi)}=2^{s-1}\,\prod_{v}\Gamma_{\xi,v}(\sigma_{v}(h_{2,v}\delta_{v})f_{v},\,\bar{\sigma}_{v}(h_{1,v}\varepsilon_{v})\bar{f}_{v};\chi_{v})\,.
\end{equation*}
Subsequently

\begin{equation*}
\begin{array}{l}\vspace{0.14in}
\Pe(\theta^{0}(\sigma(\delta) f,\weil(\delta)\phi),\chi)\,\overline{\Pe(\theta^{0}(\sigma(\varepsilon) f,\weil(\varepsilon)\phi),\chi)}
\hfill\\\vspace{0.14in}
\hspace{0.7in}=\displaystyle 2^{s-1}\,\prod_{v}\,\int\int_{(\SO(W_{1}^{\perp})_{v}\bk \SO(X)_{v})^{2}} \,(\weil_{v}(h_{2,v}\delta_{v})\phi_{v})(\xi)\,\,(\bar{\weil}_{v}(h_{1,v}\varepsilon_{v})\bar{\phi}_{v})(\xi)\hspace{0.4in}\\
\hfill\displaystyle\times\,\Gamma_{\xi,v}(\sigma_{v}(h_{2,v}\delta_{v})f_{v},\,\bar{\sigma}_{v}(h_{1,v}\varepsilon_{v})\bar{f}_{v};\chi_{v})\,dh_{1,v}\,dh_{2,v}\,.
\end{array}
\end{equation*}
In summary, we have the following formula

\begin{equation}\label{eq_placeholder}
\abs{\Pe(\theta(f,\phi),\chi)}^{2}=\displaystyle\dfrac{1}{4^{s+s\pr}}\,2^{s-1}\sum_{\delta\in\mu_{2}(F_{S\pr})}\sum_{\varepsilon\in\mu_{2}(F_{S\pr})}\displaystyle\,\prod_{v}\I_{v}(\delta_{v},\varepsilon_{v})
\end{equation}
for which we have introduced the place-holder notation

\begin{equation*}
\begin{array}{l}\vspace{0.05in}
\I_{v}(\delta_{v},\varepsilon_{v})=\displaystyle\int\int_{(\SO(W_{1}^{\perp})_{v}\bk \SO(X)_{v})^{2}} \,(\weil_{v}(h_{2,v}\delta_{v})\phi_{v})(\xi)\,\,(\bar{\weil}_{v}(h_{1,v}\varepsilon_{v})\bar{\phi}_{v})(\xi)\hspace{1in}\\
\hfill\times\,\Gamma_{\xi,v}(\sigma_{v}(h_{2,v}\delta_{v})f_{v},\,\bar{\sigma}_{v}(h_{1,v}\varepsilon_{v})\bar{f}_{v};\chi_{v})\,dh_{1,v}\,dh_{2,v}\,.
\end{array}
\end{equation*}
The $\I_{v}(\delta_{v},\varepsilon_{v})$ are connected to the local integrals of \S \ref{sec_local_calc} by

\begin{equation*}
\alpha(\theta(f_{v},\phi_{v}),\chi_{v})=\dfrac{1}{2}\,\dfrac{\zeta_{F_{v}}(2)\,\zeta_{F_{v}}(4)}{L(1,\sigma_{v},\std)}\,\sum_{\varrho_{v}\in\mu_{2}(F_{v})}\,\I_{v}(\varrho_{v},1)\,,
\end{equation*}
recalling $\alpha(\theta(f_{v},\phi_{v}),\chi_{v})=\alpha(\theta(f_{v},\phi_{v}),\theta(\bar{f}_{v},\bar{\phi}_{v});\chi_{v})$. We now separate the sum in \eqref{eq_placeholder} according to the representation $\sigma_{v}$ at $v$. The index set for the double  summation runs over $\delta,\varepsilon\in\mu_{2}(F_{S\pr})$, with $\delta=(\delta_{v})$ and $\varepsilon=(\varepsilon_{v})$, where $\delta_{v}=\varepsilon_{v}=1$ if $v\in\Set$ or $v\not\in S$.

\begin{itemize}\vspace{0.05in}
\item If $v\not\in S$ then, since $\varrho_{v}\in H(\mathcal{O}_{v})$, $\I_{v}(\varrho_{v},1)=\I_{v}(1,1)$ meaning
\begin{equation*}
\I_{v}(1,1)=\dfrac{1}{2}\,\sum_{\varrho_{v}\in\mu_{2}(F_{v})}\,\I_{v}(\varrho_{v},1) = \dfrac{L(1,\sigma_{v},\std)}{\zeta_{F_{v}}(2)\,\zeta_{F_{v}}(4)}\alpha(\theta(f_{v},\phi_{v}),\chi_{v})\,.
\end{equation*}
\vspace{0.05in}
\item If $v\in S\cap\Set$ then $\I_{v}(\iota,1)=0$. Indeed, for $f=f+0\in\V_{\sigma_{0,v}}\oplus\V_{\sigma_{0,v}^{\iota}}$ we have
\begin{equation*}
\B_{\sigma_{v}}(\sigma_{v}(\iota)f_{v},\bar{f}_{v})=\frac{1}{2}\big(\B_{\sigma_{0,v}}(0,\bar{f}_{v})+\B_{\sigma_{0,v}}(f_{v},0)\big) = 0+0\,.
\end{equation*}
The remaining term is

\begin{equation*}
\I_{v}(1,1) = 2\,\left(\dfrac{L(1,\sigma_{v},\std)}{\zeta_{F_{v}}(2)\,\zeta_{F_{v}}(4)}\right)\,\alpha(\theta(f_{v},\phi_{v}),\chi_{v})\,.
\end{equation*}
\vspace{0.05in}
\item If $v\in S\pr$ we have a four-term summation. Using that $\I_{v}(\iota,\iota)=\I_{v}(1,1)$ we find

\begin{equation*}
\displaystyle\mathop{\sum\sum}_{\delta_{v},\varepsilon_{v}\in\mu_{2}(F_{v})}\I_{v}(\delta_{v},\varepsilon_{v})\,=\,\displaystyle 2\sum_{\varrho_{v}\in\mu_{2}(F_{v})}\I_{v}(\varrho_{v},1)\,= 4\,\left(\dfrac{L(1,\sigma_{v},\std)}{\zeta_{F_{v}}(2)\,\zeta_{F_{v}}(4)}\right)\,\alpha(\theta(f_{v},\phi_{v}),\chi_{v})\,.
\end{equation*}\vspace{0.05in}
\end{itemize}
Together, these three points prove that \eqref{eq_placeholder} becomes

\begin{equation*}
\begin{array}{rcl}\vspace{0.1in}
\abs{\Pe(\theta(f,\phi),\chi)}^{2}&=&\displaystyle\dfrac{1}{4^{s+s\pr}}\,2^{s-1}\,2^{s}\,4^{s\pr}\, \left(\dfrac{L(1,\sigma,\std)}{\zeta_{F}(2)\,\zeta_{F}(4)}\right)\,\prod_{v}\,\alpha(\theta(f_{v},\phi_{v}),\chi_{v})\\
&=&\displaystyle\frac{1}{2}\left(\dfrac{L(1,\sigma,\std)}{\zeta_{F}(2)\,\zeta_{F}(4)}\right)\dfrac{\zeta_{F}(2)\,\zeta_{F}(4)\,L(\pi,\chi,1/2)}{L({\rm Ad},\pi,1)\,L(\chi_{K/F},1)}\,\prod_{v}\,\alpha^{\natural}(\theta(f_{v},\phi_{v}),\chi_{v})\,.
\end{array}
\end{equation*}

Finally, for our formula to be independent of choice of local pairings (see Remark \ref{rem_on_pairing_invariance}) we normalise the Bessel period and instead calculate
\begin{equation}\label{eq_we_should_norm}
\dfrac{\abs{\Pe(\varphi,\chi)}^{2}}{\B_{\theta(\sigma)}(\varphi,\bar{\varphi})\,\B_{\chi}(\chi,\bar{\chi})}
\end{equation}
for $\varphi\in\V_{\theta(\sigma)}$. The Petersson pairing for the one-dimensional representation $\chi$ is trivially constant in this case and is easily seen to equal the Tamagawa number
\begin{equation*}
\B_{\chi}(\chi,\bar{\chi})=\B_{\chi}(1,1)=\Vol(\Ax\Kx\bk\Ax_{K})=2\,.
\end{equation*}
The Petersson pairing for the theta lift $\theta(\sigma)$ is dealt with by the formula of Gan--Ichino \eqref{eq_GI_global_theta_pair} which states that the Petersson inner product for $\theta(\sigma)$ equals

\begin{equation*}
\B_{\theta(\sigma)}=\dfrac{L(1,\sigma,\std)}{\zeta_{F}(2)\,\zeta_{F}(4)}\,\prod_{v}\B_{\theta(\sigma_{v})}\,.
\end{equation*}
Combining these final comments gives the main result.

\begin{theorem}\label{thm_andy_thm1}

Let $(\pi,\V_{\pi})$ be an irreducible, cuspidal automorphic representation of $\PGSpf(\A)$ lifted, via the theta correspondence in \S \ref{sec_theta_section}, from (the Jacquet--Langlands transfer of) a cuspidal automorphic representation of $\GLt(\A_{E})$ with trivial central character. Let $K$ be a quadratic field extension of $F$ such that $\SO_{2}\isom\Kx/\Fx$. Let $\chi$ be a unitary Hecke character of $\Ax_{K}$ such that $\chi\vert_{\Ax}=1$; such a $\chi$ may also be viewed as an automorphic representation of $\SO_{2}(\A)$. For the cusp forms $\varphi=\otimes_{v}\varphi_{v}\in\V_{\pi}$ and $\bar{\varphi}=\otimes_{v}\bar{\varphi}_{v}\in\V_{\bar{\pi}}$ define the local integrals $\alpha^{\natural}(\varphi_{v},\chi_{v})$ as in \S \ref{sec_local}: we have $\alpha^{\natural}(\varphi_{v},\chi_{v})= 1$ for almost all $v$. For any choice of local Haar measures defining $\alpha^{\natural}(\varphi_{v},\chi_{v})$ let $C\in\C$ be the Haar measure constant (the constant of proportionality given by the ratio of the Tamagawa measure divided by the product of local measures). For each $v$, let $\B_{\pi_{v}}$ be any choice of local unitary pairing. We have proved that

\begin{equation*}
\dfrac{\abs{\Pe(\varphi,\chi)}^{2}}{\B_{\pi}(\varphi,\bar{\varphi})\,\B_{\chi}(\chi,\bar{\chi})} =\dfrac{C}{4}\,\dfrac{\zeta_{F}(2)\,\zeta_{F}(4)\,L(1/2,\pi\boxtimes\chi)}{L(1,\pi,\Ad)\,L(1,\chi_{K/F})}\,\prod_{v}\,\dfrac{\alpha^{\natural}(\varphi_{v},\chi_{v})}{\B_{\pi_{v}}(\varphi_{v},\bar{\varphi}_{v})}\,.
\end{equation*}
\end{theorem}

\begin{defn}\label{def_properly_normalised}
We define the local integrals to be \textit{properly} normalised in the following way: choose local unitary pairings $\B_{\chi_{v}}$ on each one-dimensional space $\V_{\chi_{v}}\otimes\V_{\bar{\chi}_{v}}$ such that the Petersson pairing decomposes as $\B_{\chi}=\prod_{v}\B_{\chi_{v}}$. We then take the normalised quantity 
\begin{equation*}
\B_{\chi_{v}}(\chi_{v},\bar{\chi}_{v})\alpha^{\natural}(\varphi_{v},\chi_{v})
\end{equation*}
in place of the local integrals in the formula of Theorem \ref{thm_andy_thm1}. Note that in the original definition of the local integrals (\S \ref{sec_stable}) we implicitly take $\B_{\chi_{v}}=1$ for each $v$, as per \S \ref{sec_pairings}, and we found the decomposition $\B_{\chi}=2\prod_{v}\B_{\chi_{v}}$.
\end{defn}

\begin{cor}
Assuming $C=1$, $\B_{\pi}=\prod_{v}\B_{\pi_{v}}$ and that the local integrals $\alpha^{\natural}(\varphi_{v},\chi_{v})$ are properly normalised (as in Definition \ref{def_properly_normalised}), Theorem \ref{thm_andy_thm1} becomes

\begin{equation*}
\abs{\Pe(\varphi,\chi)}^{2}=\dfrac{1}{4}\, \dfrac{\zeta_{F}(2)\,\zeta_{F}(4)\,L(1/2,\pi\boxtimes\chi)}{L(1,\pi,\Ad)\,L(1,\chi_{K/F})}\,\prod_{v}\alpha^{\natural}(\varphi_{v},\chi_{v})\,.
\end{equation*}
\end{cor}

\begin{rem}\label{rem_on_pairing_invariance}
In a more general setting, the representation $\chi$ need not be one-dimensional (when considering other groups). Normalising the left-hand-side of the equation in Theorem \ref{thm_andy_thm1} by the Petersson pairings for $\pi$ and $\chi$, and including the Haar measure constant, ensures that the local choices of pairings and measures are independent of the global setting. These objects may be chosen and may be chosen arbitrarily without affecting the formula and, in particular, the local integrals are independent of such choices (see \cite[Remark 1.3]{ichino_trilinear}).

Our normalisations may seem ad hoc at first, due to the trivial pairings on $\chi$, however we state our theorem in this way so that it sits in the more general framework of Liu's conjecture. In Liu's work one sees that the issue of normalisation appears in a natural setting and we invite the reader to check \cite[Conjecture 2.5]{liu} for consolidation.
\end{rem}

\addtocontents{toc}{\protect\setcounter{tocdepth}{-1}} 
\bibliographystyle{amsplain}			
\bibliography{bessel-yoshida}{}		
\end{document}